\documentclass[11pt,reqno]{amsart}
\usepackage[margin=1.1in]{geometry}
\usepackage{amsmath,amssymb,amsthm,mathtools}
\usepackage{booktabs,array,enumitem,longtable}

\usepackage[expansion=false]{microtype}
\usepackage[colorlinks=true,linkcolor=blue!60!black,citecolor=blue!60!black,urlcolor=blue!60!black]{hyperref}
\usepackage{xcolor}

\setlist[enumerate]{leftmargin=*,itemsep=2pt,topsep=3pt}
\setlist[itemize]{leftmargin=*,itemsep=2pt,topsep=3pt}

\theoremstyle{plain}
\newtheorem{theorem}{Theorem}[section]
\newtheorem{proposition}[theorem]{Proposition}
\newtheorem{lemma}[theorem]{Lemma}
\newtheorem{corollary}[theorem]{Corollary}
\newtheorem*{theoremA}{Theorem A}
\newtheorem*{theoremB}{Theorem B}
\newtheorem*{corollaryC}{Corollary C}
\newtheorem*{theoremD}{Theorem D}
\newtheorem*{maintheorem}{Main Theorem}
\theoremstyle{definition}

\newtheorem{convention}[theorem]{Convention}
\newtheorem{example}[theorem]{Example}
\theoremstyle{remark}
\newtheorem{remark}[theorem]{Remark}

\newcommand{\C}{\mathbb C}
\newcommand{\Q}{\mathbb Q}

\newcommand{\cA}{\mathcal A}
\newcommand{\Ae}{\mathcal A_e}
\newcommand{\cO}{\mathcal O}
\newcommand{\fm}{\mathfrak m}
\newcommand{\fc}{\mathfrak c}
\newcommand{\fp}{\mathfrak p}
\newcommand{\Hilb}{\operatorname{Hilb}}
\newcommand{\Der}{\operatorname{Der}}
\newcommand{\Hom}{\operatorname{Hom}}
\newcommand{\Ann}{\operatorname{Ann}}
\newcommand{\Res}{\operatorname{Res}}
\newcommand{\Spec}{\operatorname{Spec}}

\newcommand{\Supp}{\operatorname{Supp}}

\newcommand{\depth}{\operatorname{depth}}
\newcommand{\length}{\operatorname{length}}
\newcommand{\coker}{\operatorname{coker}}

\newcommand{\ord}{\operatorname{ord}}

\newcommand{\Id}{\mathrm{Id}}
\newcommand{\vol}{\mathrm{vol}}

\newcommand{\Aecodim}{\Ae\text{-}\mathrm{codim}}
\newcommand{\tD}{\widetilde D_X}
\newcommand{\tC}{\widetilde C_X}
\newcommand{\tT}{\widetilde T}
\newcommand{\tK}{\widetilde K}
\newcommand{\Dlog}[1]{\Der(-\log #1)}
\newcommand{\NAe}{N\!\Ae f}
\newcommand{\muI}{\mu_I}
\newcommand{\bL}{\bar\Lambda_3}
\newcommand{\bD}{\bar\Delta}

\numberwithin{equation}{section}

\title[Mond's conjecture for corank-one maps]{Mond's conjecture for corank-one maps \\ from $\mathbb C^3$ to $\mathbb C^4$}
\author{Rich\'ard Rim\'anyi}

\begin{document}

\begin{abstract}
We prove Mond's conjecture for $\mathcal A$-finite corank-one map germs $(\mathbb C^3,0)\to(\mathbb C^4,0)$: the $\mathcal A_e$-codimension is at most the image Milnor number, with equality for quasihomogeneous germs. For a quasihomogeneous germ we determine the derivations of its image modulo the conductor fields, in terms of the double-point surface, its cross-cap curve and its normalisation. This yields a closed formula, depending only on the weights and degrees, for the Hilbert series of the graded $\mathcal A_e$-normal space. Its value at $t=1$ coincides with Ohmoto's formula for the image Milnor number obtained from Segre--Schwartz--MacPherson Thom polynomials. For arbitrary germs we show that generic homogeneous corank-one germs of coprime degrees are $\mathcal A$-finite in every source dimension, and apply the reduction theorem of Fern\'andez de Bobadilla, Nu\~no-Ballesteros and Pe\~nafort Sanchis. In every dimension this reduces Mond's conjecture for corank-one germs to an inequality for generic homogeneous germs.
\end{abstract}

\maketitle

\section{Introduction}

\subsection{Motivation}

Several important results of singularity theory compare two invariants of different nature attached to the same singularity: a topological one which counts oscillations of the germ, and a deformation-theoretic one which measures how the singularity sits among its deformations. The classical example is
\[
\tau(g) \leq \mu(g)
\ \text{\cite{LS85}}
\qquad
\text{(equality if $g$ is quasi-homogeneous \cite{Gre80})}
\]
for  isolated complete intersection singularities $g$. Here $\tau$ is the Tjurina number (the dimension of the base of a miniversal deformation) and $\mu$ is the Milnor number (the number of vanishing cycles). 

Mond's conjecture is the analogue for a map germ $(\C^m,0)\to(\C^{m+1},0)$: the $\Ae$-codimension plays the role of $\tau$, and the image Milnor number that of $\mu$. It was formulated more than thirty years ago \cite{Mon91} and remains open for $m\ge3$. See \cite{NBPS24} for a recent excellent survey on the topic.

In recent work with J.~Koncki \cite{KR}, the author completed a programme initiated by Ohmoto \cite{Ohm16} and continued by Pallar\'es and Pe\~nafort Sanchis \cite{PPS21}. The result computes the image Milnor number of a quasihomogeneous germ from its weights and degrees, using Segre--Schwartz--MacPherson Thom polynomials of multisingularities. It is hence natural to look for an approach to the $\Ae$-codimension side with the same structure. For the time being we have not found one in general, but the attempt led to the present paper.

\subsection{Results}

Let $f:(\C^m,0)\to(\C^{m+1},0)$ be a holomorphic map germ of finite $\Ae$-codimension. In Mather's nice dimensions \cite{Mat71} ($m\leq 14$), and for corank-one germs in every dimension (Remark~\ref{rem:nicedims}), a stable perturbation of $f$ has image homotopy equivalent to a bouquet of $m$-spheres \cite{Mon91}. Their number is the \emph{image Milnor number} $\muI(f)$ (for corank-one germs it can be computed from the multiple-point spaces \cite{GM93}). Mond conjectured that
\[
\Aecodim(f)\le\muI(f),
\]
with equality if $f$ is quasihomogeneous \cite{Mon91,MNB20}.\footnote{Some authors also include the converse, that equality holds only for quasihomogeneous germs, and the conjecture is also formulated for multigerms. This paper is concerned with monogerms in the form stated.} 
The conjecture compares the number of parameters in a versal deformation with the vanishing topology of the image.

For $m=1$ the conjecture is a theorem of Mond \cite{Mon95}, and for $m=2$ it follows from work of de Jong and van Straten \cite{dJvS91} and Mond \cite{Mon91}. For $m\ge3$ it was known only in special classes: germs of $\Ae$-codimension one \cite{CMWA02}, and, for the inequality, augmentations \cite{BROS25}. The weak form ``$\muI(f)=0$ if and only if $f$ is stable'' holds for germs with isolated instability \cite{GCNB}. For weighted homogeneous $f$ in the nice dimensions, Fern\'andez de Bobadilla, Nu\~no-Ballesteros and Pe\~nafort Sanchis proved the reverse inequality $\muI(f)\le\Aecodim(f)$ \cite[\S3 and Cor.~6.3]{FdBNBPS}. Thus, in the quasihomogeneous case, it remains to prove $\Aecodim(f)\le\muI(f)$. They also reduced the general conjecture to quasihomogeneous germs of large multiplicity \cite[Thms~7.1, 7.2]{FdBNBPS}. We prove the following case.

\begin{maintheorem}
Every $\Ae$-finite corank-one germ $f:(\C^3,0)\to(\C^4,0)$ satisfies
\[
\Aecodim(f)\le\muI(f),
\]
with equality if $f$ is $\cA$-equivalent to a quasihomogeneous germ.
\end{maintheorem}

The algebraic part of the proof concerns a quasihomogeneous germ with image $X$ and normalisation $f:S\to X$. Write $\fc$ for the conductor of $\cO_X\subset\cO_S$, and $D_X$ for the double-point surface. We first identify $\fc$ with the reduced double-point ideal in the source. Restricting derivations to $D_X$ then gives
\[
\Der(\cO_X)/\fc\theta_S\cong\Der(\cO_{D_X};C_X),
\]
where the right-hand side consists of derivations tangent to the cross-cap curve $C_X$. We have therefore reduced the calculation on the three-dimensional image to one on the surface $D_X$, with the curve $C_X$ distinguished.

The normalisation $\tD=D^2(f)/S_2$ of that surface is a complete intersection. Its branch curve has a reduced equation $\bD$, and its logarithmic derivations have a basis consisting of the Euler field $E$ and a cross-product field $\zeta$. The identity
\[
E\wedge\zeta=\pm v\bD\,\varepsilon
\]
identifies this basis away from the vertex. The descent condition at a triple point is exactly the triple-point conductor condition on the transverse part. Hence
\[
\Der(\cO_{D_X};C_X)
=\cO_{D_X}E\oplus\bL\cO_{\tD}\zeta.
\]
The local calculations extend across the vertex because the proposed submodules have depth two. Their Hilbert series determine the graded $\Ae$-normal space, and its dimension agrees with Ohmoto's formula for $\muI$ \cite{Ohm16}. No genericity assumption on the quasihomogeneous germ is needed.

To pass from the quasihomogeneous case to arbitrary corank-one germs, we construct quasihomogeneous $\Ae$-finite germs with arbitrarily large orders of their nonlinear components: a generic pair of homogeneous polynomials of coprime degrees has this property. The main point is to prove the expected smoothness of the multiple-point spaces. For this we use interpolation for divided differences and generic smoothness, followed by the Marar--Mond criterion. The fixed-corank reduction theorem of Fern\'andez de Bobadilla, Nu\~no-Ballesteros and Pe\~nafort Sanchis \cite[Thm~7.2]{FdBNBPS} then proves the Main Theorem. In every dimension the same construction reduces the corank-one conjecture to an upper bound for the $\Ae$-codimension of a single explicit family of homogeneous germs (Theorem~D in Section~\ref{sec:setting}).

The method is related to the study of liftable vector fields for stable corank-one maps \cite{HL}, their behaviour under unfoldings \cite{NBOS}, and their relation to quasihomogeneity \cite{BROS}. Our description uses the conductor and the geometry of target multiple-point spaces \cite{MP89,NBPSV}. The quasihomogeneous image Milnor number formulas of Mond and Ohmoto \cite{MonQH91,Ohm16}, and their higher-dimensional extension in \cite{KR}, provide the topological comparison.

Section~\ref{sec:setting} states the main results. After the local models and conductor calculation, Section~\ref{sec:thmA} treats the two-dimensional case as a warm-up. Its proof is not needed for the three-dimensional argument in Sections~\ref{sec:rigidity}--\ref{sec:normalisation}. Section~\ref{sec:proofB} computes the Hilbert series and includes a worked example. Section~\ref{sec:beyond} proves the reduction to arbitrary germs (Theorem~D). Section~\ref{sec:outlook} discusses the case $m=4$ and other open questions. Appendix~\ref{sec:algebra} records the elementary algebra used in the Hilbert-series calculations.

\subsection{Acknowledgments}
The author is grateful to J. Koncki, T. Ohmoto, J. J. Nu\~no-Balles\-teros, and G. Pe\~nafort Sanchis for helpful conversations on the topic, and their encouragements. Computer experiments that led to the statements of Theorems~A and~B, as well as first drafts of several proofs and independent audits of the arguments were produced with substantial help from AI (Anthropic's Claude; OpenAI's ChatGPT). The author takes full responsibility for the content.

\section{Setting and main results}\label{sec:setting}

For $m\in \{2,3\}$ consider a quasihomogeneous corank-one germ
\[
f: (\C^m,0) \to (\C^{m+1},0),
\qquad
f(x,y)=(x,p(x,y),q(x,y)),
\]
where $x=x_1$ or $x=(x_1,x_2)$ depending on $m=2$ or $3$. Here $x_i$ has weight $a_i\ge1$, $y$ has weight $v\ge1$, $p,q\in\C[x,y]$ are weighted homogeneous of degrees $\mu v$ and $k$, with
\[
p(0,y)=c\,y^{\mu}\ \ (c\ne0),\qquad \mu\ge2,\qquad k>v .
\]
($p(0,y)\ne0$ is a normalisation, see Remark~\ref{rem:normalisation}.) Every $\Ae$-finite quasihomogeneous corank-one germ takes this form after weighted-homogeneous changes of coordinates in source and target (Remarks~\ref{rem:normalform}, \ref{rem:normalisation} and~\ref{rem:kv}).

We write $S=\C^m$ (source), $T=\C^{m+1}$ with coordinates $(X_1,\dots,X_{m-1},Y,Z)$ of weights $(a,\mu v,k)$, $\cO_S=\C[x,y]$, $\cO_T=\C[X,Y,Z]$, $f^*:\cO_T\to\cO_S$, $X_i\mapsto x_i$, $Y\mapsto p$, $Z\mapsto q$. All rings are positively graded with degree-$0$ part $\C$. By ``module'' we mean finitely generated graded module unless said otherwise, and define $\Hilb M=\sum_\delta\dim_\C M_\delta\,t^\delta$.

The Euler fields are $\eta_E=\sum a_ix_i\partial_{x_i}+vy\partial_y$ on $S$ and $E=\sum a_iX_i\partial_{X_i}+\mu vY\partial_Y+kZ\partial_Z$ on $T$, and we have $df(\eta_E)=E\circ f$. On any positively graded ring $R$ considered below we also write $E$ for the \emph{Euler derivation} $h\mapsto(\deg h)\,h$ ($h$ homogeneous). On $\cO_T$, $\cO_X$, $\cO_S$, $\cO_{D_X}$, $\cO_{\tD}$ these are the fields just defined, or their unique extensions (Lemma~\ref{lem:seidenberg-graded}).

We use the usual notation of singularity theory:
$\theta_S=\Der_\C\cO_S$, $\Theta_T=\Der_\C\cO_T$, $\theta(f)$ the module of vector fields along $f$, $tf:\theta_S\to\theta(f)$, $\eta\mapsto(\eta(f_1),\dots,\eta(f_{m+1}))$, $\omega f:\Theta_T\to\theta(f)$, $\xi\mapsto(\xi_j\circ f)_j$, and
\[
\NAe=\theta(f)\big/\big(tf(\theta_S)+\omega f(\Theta_T)\big),\qquad \Aecodim(f)=\dim_\C\NAe .
\]
For the grading, put $w=(a_1,\dots,a_{m-1},v)$ and $d=(a_1,\dots,a_{m-1},\mu v,k)$. Our convention is $M(b)_\delta=M_{\delta+b}$, so a coordinate field $\partial_z$ has degree $-\deg z$, and
\[
\theta_S=\bigoplus_{i=1}^{m}\cO_S(w_i),\qquad
\Theta_T=\bigoplus_{j=1}^{m+1}\cO_T(d_j),\qquad
\theta(f)=\bigoplus_{j=1}^{m+1}\cO_S(d_j).
\]
Thus $tf$ and $\omega f$ have degree zero. Unless explicitly stated otherwise, through Section~\ref{sec:proofB} we assume $m\in\{2,3\}$ and that $f$ is quasihomogeneous and $\Ae$-finite. For finite germs, $\cA$-finiteness (finite determinacy) is equivalent to finite $\Ae$-codimension \cite[Thm~6.1]{MNB20}, and we use the two terms interchangeably. Let $X:=f(S)$ be the image of $f$ and $\cO_X:=f^*\cO_T\subset\cO_S$.

Since all weights are positive and $p(0,y)=cy^\mu$, the map $f$ is finite and $f^{-1}(0)=\{0\}$: indeed $f(x,y)=0$ forces $x=0$ and then $p(0,y)=cy^\mu=0$. Consequently ``off the origin'' has the same meaning in the source and in the target, and $\cO_S$ is a finite $\cO_T$-module (\S\ref{sec:image}).

We distinguish source fields from liftable target fields. Set
\[
L:=\{\eta\in\theta_S:tf(\eta)\in\omega f(\Theta_T)\},\qquad
\operatorname{Lift}(f):=\{\xi\in\Theta_T:\omega f(\xi)\in tf(\theta_S)\}.
\]
If $G=0$ is the reduced equation of $X$, then Proposition~\ref{prop:L} identifies
\[
L\cong\Der(\cO_X)\cong\operatorname{Lift}(f)/(G\Theta_T),
\qquad \operatorname{Lift}(f)=\Dlog X.
\]
We use this identification to regard $\Der(\cO_X)$ as a submodule of $\theta_S$.

For compact Hilbert-series statements, write $\lambda=(\mu-1)(k-v)$ and
\begin{gather*}
H_S=\frac{1}{\prod_{i=1}^{m-1}(1-t^{a_i})(1-t^v)},\qquad
H_X=\frac{1-t^{\mu k}}{\prod_{i=1}^{m-1}(1-t^{a_i})(1-t^{\mu v})(1-t^k)},\\
A_S=\sum_{i=1}^{m-1}t^{-a_i}+t^{-v},\qquad
A_T=\sum_{i=1}^{m-1}t^{-a_i}+t^{-\mu v}+t^{-k},\\
\mathcal H_m=(A_T-A_S)H_S+(1-A_T)H_X+t^\lambda(A_S-1)H_S.
\end{gather*}

\begin{theoremA}[$m=2$]
Let $\Lambda\in\cO_S$ be the double-point function of Section~\ref{sec:double}, of degree $\lambda$. Then
\[
\Der_\C(\cO_X)=\cO_X\eta_E+\Lambda\theta_S,\qquad
\cO_X\eta_E\cap\Lambda\theta_S=\Lambda\cO_S\eta_E.
\]
Consequently $\Hilb_t\NAe=\mathcal H_2(t)$.
\end{theoremA}

\begin{theoremB}[$m=3$]
Let $\fc=\Lambda\cO_S$ be the conductor of $\cO_X\subset\cO_S$, $D_X=\Spec(\cO_X/\fc)$ the double-point surface, $C_X\subset D_X$ the reduced cross-cap curve, and $\tD=D^2(f)/S_2$ its normalisation. Let $\bL$ be the triple-point function on $\tD$ of Section~\ref{sec:triple}, of degree $(\mu-2)(k-2v)$, and $\zeta$ the cross-product field of Section~\ref{sec:freeness}, of degree $e_\zeta=k+(\mu-3)v-a_1-a_2$. Then
\[
\Der_\C(\cO_X)/\fc\theta_S
\cong\Der(\cO_{D_X};C_X)
=\cO_{D_X}E\oplus\bL\cO_{\tD}\zeta.
\]
Here $\bL$ is a unit when the triple curve is empty. Consequently, with $d_0=\lambda-a_1-a_2$, one has
\[
\Hilb_t\NAe=\mathcal H_3(t)+t^{d_0}H_{\tD}(t),\qquad
H_{\tD}=\frac{(1-t^{(\mu-1)v})(1-t^{k-v})}
{(1-t^{a_1})(1-t^{a_2})(1-t^v)(1-t^{2v})}.
\]
\end{theoremB}

The notation $\Der(\cO_{D_X};C_X)$ means derivations preserving the reduced ideal of $C_X$. Theorem~B describes the quotient by conductor fields. It does not assert that $\Der(\cO_X)$ itself splits. The expanded finite-sum forms of both Hilbert series are given in Section~\ref{sec:expanded}.

\begin{corollaryC}
For $m\in\{2,3\}$, the graded $\Ae$-normal space has Hilbert series depending only on the weights and degrees, and
\[
\Aecodim(f)=\muI(f)=R(w,d),
\]
where $R(w,d)$ is displayed in Section~\ref{sec:proofC}. The formula for $\muI$ is due to Mond for $m=2$ \cite{MonQH91} and Ohmoto for $m=3$ \cite[Thm~6.20, Ex.~6.22]{Ohm16}. See also \cite{PPS21}, \cite[Thms~2.4--2.5]{NBPS24} and \cite[Thm~9.4]{KR}.
\end{corollaryC}

By \cite[Thm~6.7]{FdBNBPS}, Corollary~C also shows that these germs have the Cohen--Macaulay property of \cite{FdBNBPS} (Corollary~\ref{cor:CM}).

The passage to arbitrary germs rests on the following result, proved in Section~\ref{sec:beyond} (Proposition~\ref{prop:unbounded} and Corollary~\ref{cor:reduction}).

\begin{theoremD}
Let $m\ge2$, and let $\mu,k\ge2$ be coprime integers. For a generic pair of homogeneous polynomials $p,q\in\C[x_1,\dots,x_{m-1},y]$ of degrees $\mu$ and $k$, the germ $(x_1,\dots,x_{m-1},p,q)\colon(\C^m,0)\to(\C^{m+1},0)$ is $\Ae$-finite. Moreover, if, for infinitely many $M$, some $\Ae$-finite germ $h_M$ of this form with $(\mu,k)=(M,M+1)$ satisfies $\Aecodim(h_M)\le\muI(h_M)$, then Mond's conjecture holds for all $\Ae$-finite corank-one germs $(\C^m,0)\to(\C^{m+1},0)$.
\end{theoremD}

For $m=3$ the hypothesis holds for every $M$ by Corollary~C, and this proves the Main Theorem. For each $m\le14$, $\muI(h_M)$ is an explicit function of $M$ by \cite[Thm~9.4]{KR}, so in every pair of nice dimensions the corank-one conjecture is reduced to an explicit upper bound for the $\Ae$-codimension of one family of homogeneous germs.

\begin{remark}[no genericity]\label{rem:nogenericity}
Theorems~A and~B require no genericity of the coefficients of $p,q$: only $\Ae$-finiteness (through stability off the origin) and quasihomogeneity are used. Hence the graded $\Ae$-codimension is constant on the $\Ae$-finite locus of the family with given $(w,d)$.
\end{remark}

\begin{remark}[the normalisation $p(0,y)\ne0$]\label{rem:normalisation}
For a finite corank-one germ $(x,p,q)$, $p(0,y)$ and $q(0,y)$ are not both zero. If $p(0,y)=0$, then $q(0,y)=c'y^{k/v}$, and exchanging the roles of $p$ and $q$ puts the germ in the normalised form with $\mu'=k/v$, $k'=\mu v$. Theorems~A, B and Corollary~C apply to the exchanged germ. Since $\NAe$ and $\muI$ are $\cA$-invariants, Corollary~C covers all $\Ae$-finite quasihomogeneous corank-one germs with $m\in\{2,3\}$ (the graded law then has the exchanged parameters).
\end{remark}

\begin{remark}[the hypothesis $k>v$]\label{rem:kv}
If $k<v$ then $q=q(x)$ does not involve $y$ (no monomial of weight $k<v$ contains $y$), so $f$ is $\cA$-equivalent to $(x,p,0)$, whose non-immersive locus $V(p_y)$ has dimension $m-1\ge1$. At a non-immersive point off the origin the germ is not stable (the only stable non-immersive corank-one germs are cross-caps, whose image is not contained in a smooth hypersurface), contradicting $\Ae$-finiteness by (MG) below. If $k=v$ then $q=\alpha y+q'(x)$. For $\alpha=0$ the same argument applies, and for $\alpha\ne0$ the germ $f$ is an immersion, hence stable, $\NAe=0$, and there is nothing to prove. Thus $k>v$ is the only case of interest.
\end{remark}

\begin{remark}[normal form]\label{rem:normalform}
Let $g\colon(\C^m,0)\to(\C^{m+1},0)$ be quasihomogeneous of corank one, i.e.\ weighted homogeneous with respect to positive weights $w'_1,\dots,w'_m$ of the source coordinates $z$ and degrees $d'_1,\dots,d'_{m+1}$ of its components. Then $dg(0)$ preserves the induced gradings of the tangent spaces, so after graded linear changes of coordinates in source and target we may assume that $g_i=z_i+h_i$ for $i\le m-1$, with $h_i\in\fm^2$ weighted homogeneous of degree $d'_i=w'_i$, and that $g_m,g_{m+1}\in\fm^2$. The substitution $x_i:=g_i$ ($i\le m-1$), $y:=z_m$ is a weighted-homogeneous change of coordinates with linear part the identity, and in these coordinates $g=(x,p,q)$ with $p,q$ weighted homogeneous. If $g$ is $\Ae$-finite, it is finite, so $p(0,y)$ and $q(0,y)$ are not both zero. After the exchange of Remark~\ref{rem:normalisation} if necessary we may assume $p(0,y)=cy^\mu$ with $c\ne0$, and by Remark~\ref{rem:kv} we may assume $k>v$. This is the form considered above.
\end{remark}

\begin{remark}[symmetry in $p$ and $q$]\label{rem:symmetry}
The Hilbert series of Theorems~A and~B depend on $\deg p=\mu v$ and $\deg q=k$ only through $t^{\mu v}$, $t^{k}$, $t^{\mu k}=t^{\mu v\cdot k/v}$, $t^{(\mu-1)v}=t^{\mu v-v}$, $t^{k-v}$ and $t^{\lambda}=t^{(\mu v-v)(k-v)/v}$, so they are symmetric under exchanging $\deg p$ and $\deg q$. This is consistent with Remark~\ref{rem:normalisation}: when both $p(0,y)$ and $q(0,y)$ are nonzero, either of them may be used for the normalisation, and the two applications of Theorem~A or~B, which involve different double-point and triple-point functions, give the same answer. At $t=1$ it is the symmetry of \eqref{eq:R2} and \eqref{eq:R3} under $(\mu,k)\mapsto(k/v,\mu v)$.
\end{remark}

\section{Preliminaries and external inputs}\label{sec:prelim}

\subsection{Graded modules, analytic germs, and the comparison lemma}\label{sec:comparison}

All our rings are finitely generated positively graded $\C$-algebras $R=\bigoplus_{n\ge0}R_n$ with $R_0=\C$. We define the irrelevant ideal $\fm=\bigoplus_{n>0}R_n$, and $U:=\Spec R\setminus\{\fm\}$ the \emph{punctured cone}. The $\C^*$-action on $\Spec R$ has $\{\fm\}$ as its unique fixed point, and $U$ is a union of orbits. Depth is always taken with respect to $\fm$, equivalently in the localisation $R_\fm$.

\subsubsection*{Analytic versus algebraic $\NAe$.}
The germ $f$ is usually treated with analytic local rings $\cO^{an}_S=\C\{x,y\}$, $\cO^{an}_T=\C\{X,Y,Z\}$. Since $p,q$ are polynomials, $tf$ and $\omega f$ are defined over the polynomial rings, and $N^{an}=\NAe\otimes_{\cO_T}\cO^{an}_T$. All modules are $\cO_T$-modules via $f^*$, and they are finitely generated since $f$ is finite. Base change to the flat ring $\cO^{an}_T$ commutes with cokernels. We also use $\cO_S\otimes_{\cO_T}\cO^{an}_T\cong\cO^{an}_S$, which holds because $f$ is finite with $f^{-1}(0)=\{0\}$. If $N^{an}$ is finite-dimensional (i.e.\ $f$ is $\Ae$-finite as an analytic germ) then $\NAe$ is supported at $\fm$: its support is a $\C^*$-invariant closed subset of $\Spec\cO_T$, hence a union of orbit closures all passing through the origin, and if it contained a point $\ne0$ then $N^{an}$ would be nonzero at points arbitrarily close to $0$, contradicting $\dim N^{an}<\infty$. A module supported at $\fm$ is unchanged by completion, so $\NAe=N^{an}$, and the graded polynomial module computes the $\Ae$-codimension and its grading. Conversely finite-dimensionality of the graded module implies that of $N^{an}$. We therefore work with graded polynomial modules throughout.

\begin{lemma}[comparison lemma]\label{lem:comparison}
Let $R$ be as above and $M\subseteq N$ finitely generated $R$-modules  (\emph{not} necessarily graded). Suppose that for every closed point $P\in U$ the analytic local ring $\cO_P^{an}$ of $\Spec R$ at $P$ satisfies $M\otimes_R\cO^{an}_P=N\otimes_R\cO^{an}_P$ as submodules of $N\otimes_R\cO^{an}_P$. Then $N/M$ is supported at $\{\fm\}$, i.e.\ $\widetilde M|_U=\widetilde N|_U$, and $N/M$ has finite length.
\end{lemma}

\begin{proof}
We have that $R_P\to\cO^{an}_P$ is a flat local homomorphism of local rings \cite[\S2, Prop.~3 and Cor.~1]{Ser56} (the two rings have the same completion, and a pair of local rings with the same completion is a flat couple. See also \cite[App.~B]{Har77}), hence faithfully flat. Therefore $(N/M)\otimes_R\cO^{an}_P=0$ forces $(N/M)_P=0$. The support of $N/M$ is closed and contains no closed point of $U$. Since $R$ is a Jacobson ring every nonempty closed subset of $\Spec R$ other than $\{\fm\}$ contains a closed point of $U$, so $\Supp(N/M)\subseteq\{\fm\}$ and $N/M$ is annihilated by a power of $\fm$, hence of finite length.
\end{proof}

In our applications $M,N$ are obtained as kernels, annihilators, conductors, or intersections. These constructions commute with flat base change. Reduced structures and normalisations are also compatible with $R_P\to\cO^{an}_P$: the two rings have the same completion, reducedness and normality pass from the excellent ring $R_P$ to its completion \cite[IV$_2$, \S7.8]{EGA}, and they descend from the completion to $\cO^{an}_P$ along the faithfully flat map $\cO^{an}_P\to\widehat{\cO^{an}_P}$ \cite[\S23]{Mat86}. So ``$M\otimes\cO^{an}_P$'' is the analytic version of the same object at $P$, and it may be computed in any analytic coordinates, in particular in the normal forms of~\S\ref{sec:local}. Moreover every object we consider is an invariant of the map germ $f$ up to $\cA$-equivalence (see~\S\ref{sec:local0}), so the normal form of the \emph{multigerm} of $f$ over $P$ may be used. The lemma is also applied to modules that are not graded, such as $\cO_X+\cO_Xh$ with $h$ not necessarily homogeneous (Lemmas~\ref{lem:seminormal} and~\ref{lem:conductorD}). This is why no grading is assumed.

\begin{lemma}[depth two and the punctured cone]\label{lem:depth2}
Let $M$ be a graded $R$-module.
\begin{enumerate}[label=\textup{(\alph*)}]
\item There is an exact sequence $0\to H^0_\fm(M)\to M\to\Gamma(U,\widetilde M)\to H^1_\fm(M)\to0$ \cite[Ex.~III.2.3(e)]{Har77}, \cite{Gro67}, compatible with gradings.
\item If $\depth_\fm M\ge2$ then $H^0_\fm(M)=H^1_\fm(M)=0$ \cite[Thm~3.5.7]{BH98}, so $M=\Gamma(U,\widetilde M)$.
\item If $M\subseteq N$ are graded modules with $\depth M\ge2$, $\depth N\ge1$ and $\widetilde M|_U=\widetilde N|_U$, then $M=N$.
\end{enumerate}
\end{lemma}

\begin{proof}[Proof of \textup{(c)}]
Since $N/M$ is supported at $\fm$, we have $H^0_\fm(N/M)=N/M$. The local cohomology sequence of $0\to M\to N\to N/M\to0$ gives $H^0_\fm(N)\to H^0_\fm(N/M)\to H^1_\fm(M)$ exact. The depth assumptions give $H^0_\fm(N)=H^1_\fm(M)=0$, so $N/M=0$.
\end{proof}

In applications, $N$ is a submodule of a module of positive depth, so $\depth N\ge1$ is automatic. It therefore suffices to exhibit $M\subseteq N$, prove $\depth M\ge2$, and check equality away from the vertex.

\subsection{Depth facts}\label{sec:depth}

For finitely generated graded modules over $R$ we have

\begin{enumerate}[label=\textup{(D\arabic*)}]
\item\label{D1} \textbf{Depth lemma} \cite[Prop.~1.2.9]{BH98}. For $0\to M'\to M\to M''\to0$:
\[
\begin{array}{lcl}
    \depth M & \ge & \min(\depth M',\depth M''),\\
    \depth M'& \ge & \min(\depth M,\depth M''+1),\\
    \depth M'' & \ge& \min(\depth M'-1,\depth M).
\end{array}
\]

\item\label{D2} \textbf{Submodules of depth-positive modules.} If $\depth F\ge1$ then every submodule $A\subseteq F$ has depth $\ge1$ (with $\depth 0=\infty$): an $F$-regular element of $\fm$ is injective on $A$. In particular submodules of a free module over a ring of positive depth, and submodules of a reduced ring of positive dimension (which has depth $\ge1$, since $\fm$ is not a minimal prime and a reduced ring has no embedded primes), have depth $\ge1$. A nonzero module of finite length has depth $0$. \emph{No domain or torsion-freeness hypothesis is needed.}
\item\label{D3} \textbf{Dual modules.} If $\depth N\ge2$ then $\depth\Hom_R(M,N)\ge2$ for every finitely generated $M$ \cite[Ex.~1.4.19]{BH98}. Proof: a presentation $R^b\to R^a\to M\to0$ gives $0\to\Hom(M,N)\to N^a\to I\to0$ with $I\subseteq N^b$; $\depth I\ge1$ by \ref{D2}, and \ref{D1} gives $\depth\Hom(M,N)\ge\min(\depth N^a,\depth I+1)\ge2$.
\item\label{D4} \textbf{Cohen--Macaulay facts} \cite[Ch.~1--2]{BH98}. A quotient of a polynomial ring by a regular sequence (complete intersection) is Cohen--Macaulay and Gorenstein; an ideal generated by $r$ elements of height $r$ in a CM ring is generated by a regular sequence. A hypersurface section $R/(h)$ of a CM ring $R$ by a nonzerodivisor $h$ is CM. A CM ring has no embedded primes. Hence a CM ring which is reduced at its minimal primes is reduced. In a normal domain every principal ideal $(h)\ne0$ is unmixed of height one, $(h)=\bigcap_{\fp}\fp^{(\ord_{\fp}h)}$ over the height-one primes. In particular if $\ord_{\fp}(h)\le1$ for all $\fp$ then $(h)$ is the reduced ideal of $V(h)$.
\item\label{D5} \textbf{$a$-invariant} \cite[Def.~3.6.13, Cor.~3.6.14 and the remark after Thm~3.6.19]{BH98}, \cite{GW78}. If $R$ is graded CM of dimension $d$ with Hilbert series $H_R(t)$, then
\[a(R):=\max\{n:H^d_\fm(R)_n\ne0\}=\deg H_R(t)\]
(degree of the rational function $=\deg$ numerator $-\deg$ denominator). (Proof from the cited statements: for a homogeneous system of parameters $x_1,\dots,x_d$ of degrees $\alpha_i$, which is an $R$-sequence, $a(R/xR)=a(R)+\sum\alpha_i$ by Cor.~3.6.14. For the zero-dimensional ring $R/xR$, $a(R/xR)=\max\{n:(R/xR)_n\ne0\}=\deg H_{R/xR}(t)$; and $H_R(t)=H_{R/xR}(t)/\prod(1-t^{\alpha_i})$.) In particular, for $d=1$: $H^1_\fm(R)_n=0$ for $n>\deg H_R(t)$, and by Lemma~\ref{lem:depth2}(a) $R_n=\Gamma(U,\cO)_n$ for $n>\deg H_R(t)$.
\end{enumerate}

\textbf{Depth over a finite extension.} If $R\to R'$ is a finite graded homomorphism of graded rings as above (e.g.\ $\cO_X\to\cO_S$, $\cO_{D_X}\to\cO_{\tD}$ below) then $\sqrt{\fm_RR'}=\fm_{R'}$, local cohomology is unchanged on replacing an ideal by its radical, and hence for a finitely generated $R'$-module $M$ one has $\depth_{\fm_R}M=\depth_{\fm_{R'}}M$ \cite[Ex.~1.2.26]{BH98}. So a CM $R'$-module of dimension $d$ is a CM $R$-module of dimension $d$ (``maximal CM'' if $d=\dim R$).

\subsection{Saito's criterion}\label{sec:saito}

\begin{enumerate}[label=\textup{(S)}]
\item \cite[(1.8)(ii)]{Sai80} Let $(M,z)$ be a complex manifold of dimension $n$ with coordinates $z_1,\dots,z_n$ near a point, $D\subset M$ a reduced divisor with local equation $h$, and $\delta_1,\dots,\delta_n\in\Dlog D$ (germs of vector fields with $\delta_i(h)\in(h)$). Then $\{\delta_i\}$ is a free basis of $\Dlog D$ near the point if and only if $\det(\delta_i(z_j))=u\,h$ with $u$ a unit.
\end{enumerate}

\begin{remark}
The application in Proposition~\ref{prop:freeness} is to a smooth curve on a smooth surface. An elementary proof is included there. The bases used in Proposition~\ref{prop:nc}(c) and Lemma~\ref{lem:normalisation}(d) are also verified directly.
\end{remark}

\subsection{Seidenberg's theorem}\label{sec:seidenberg}

\begin{enumerate}[label=\textup{(Se)}]
\item \cite[\S3, Theorem p.~168 and Corollary p.~169]{Sei66} Let $A$ be a noetherian domain containing $\Q$ and $\bar A$ its integral closure in its field of fractions. Every derivation $\delta:A\to A$ extends uniquely to a derivation $\bar A\to\bar A$. In particular every derivation of $A$ maps the conductor of $A\subset\bar A$ into itself.
\end{enumerate}
We use it for $A=\cO_X$ (a domain, \S\ref{sec:image}) with $\bar A=\cO_S$, and for $A=\cO_{D_X}$ (a domain, Lemma~\ref{lem:normalisation}(b)) with $\bar A=\cO_{\tD}$ (\S\ref{sec:normalisation}). The reduced-ring version of Seidenberg's theorem is not needed. In the analytic local rings at points of the punctured cone the relevant rings may be reducible, but there all derivation modules are computed explicitly (\S\ref{sec:local}).

\begin{lemma}\label{lem:seidenberg-graded}
In the situation of \textup{(Se)}, if $A$ and $\bar A$ are graded (with $A\subseteq\bar A$ a graded inclusion) and $\delta$ is homogeneous of degree $\delta_0$, then its extension is homogeneous of degree $\delta_0$. In particular the Euler derivation of $A$ extends to the Euler derivation of $\bar A$.
\end{lemma}
\begin{proof}
For $\lambda\in\C^*$ let $\lambda^*$ denote the grading automorphism $h\mapsto\lambda^{\deg h}h$ of $\bar A$, which preserves~$A$. Then $\lambda^{\delta_0}(\lambda^*)^{-1}\bar\delta\lambda^*$ is a derivation of $\bar A$ extending $\delta$, hence equals $\bar\delta$ by uniqueness. This is homogeneity. The Euler derivations of $A$ and $\bar A$ agree on $A$, so uniqueness gives the last claim.
\end{proof}

\subsection{Stability off the origin and stable multigerms}\label{sec:stable}

\begin{enumerate}[label=\textup{(MG)}]
\item \textbf{Mather--Gaffney criterion} \cite[Thm~4.5]{MNB20} (also \cite[Thm~2.11]{MM89}, attributed there to Gaffney and to \cite[Thm~2.1]{Wal81}). A germ $f:(\C^n,S')\to(\C^p,0)$, $S'$ a finite set, has finite $\Ae$-codimension if and only if it has a representative $f:\mathcal U\to\mathcal V$ such that for every $P\in\mathcal V\setminus\{0\}$ the multigerm of $f$ at $f^{-1}(P)\cap\mathcal U$ is stable (``isolated instability'').
\end{enumerate}

In the quasihomogeneous case the set of points $P\in X$ at which the multigerm is \emph{not} stable is $\C^*$-invariant and closed (by Mather's infinitesimal criterion it is the support of the finitely generated $\cO_T$-module $\NAe$ of \S\ref{sec:comparison}), hence, being contained in $\{0\}$ near $0$, equal to $\{0\}$ (or empty) globally on the affine cone. Hence we have that {\em the multigerm of $f$ over every point $P\in X\setminus\{0\}$ is stable.} Likewise every $\cA$-invariant stratum of $X$ (e.g. multisingularity loci, like double points, triple points, cross-caps, etc.)\ is a $\C^*$-invariant subvariety.

\begin{enumerate}[label=\textup{(St)}]
\item The list of {\em stable} multisingularities $(\C^m,0)\to (\C^{m+1},0)$ for $m\leq 3$ is well known (Morin \cite{Mor65} for the monogerms; e.g.\ \cite[Thm~7.1, Cor.~4.7 and \S4.3]{MNB20}, or \cite{TPP})
\[
\begin{array}{ccl}
    \text{source codimension } 0: & & A_0 \\
    \text{source codimension } 1: & & A_0^2 \\
    \text{source codimension } 2: & & A_0^3, A_1 \\
    \text{source codimension } 3: & & A_0^4, A_0A_1
\end{array}
\]
Here $A_0$ is the immersion germ and $A_1$ is the cross-cap, which, following Mond, we denote by
\[
S_0:\quad(x_1,\dots,x_{m-1},y)\longmapsto(x_1,\dots,x_{m-1},\,y^2,\,x_1y);
\]
accordingly we write $S_0A_0$ for the multigerm $A_0A_1$. The next Morin singularity, $A_2$, first occurs for $m\ge4$.
A multigerm is stable if and only if each branch is stable and the analytic strata $\tau(f_i)$ of the branches meet in general position \cite[Thm~3.3]{MNB20}, \cite[Thm~2.12]{MM89}. For immersed branches this means that their images are transverse smooth hypersurfaces, which can be straightened simultaneously to coordinate hyperplanes.
\end{enumerate}

\subsection{Multiple point spaces}\label{sec:multiple}

For a corank-one germ $f=(x,p,q)$ we use the usual divided difference notation, as follows. For a polynomial $h$, we define $h[i]=h(x,y_i)$ and for $j\geq 2$
\[
h[i_1,\ldots,i_j]=\frac{h[i_1,\ldots,i_{j-1}]-h[i_1,\ldots,i_{j-2},i_j]}{y_{i_{j-1}}-y_{i_j}}.
\]
This is a polynomial depending only on the set $I=\{i_1,\ldots,i_j\}$, symmetric in the $y$ variables.
The $j$-th \emph{multiple point scheme} is
\[
D^j(f)=V\big(p[1,\dots,i],\,q[1,\dots,i]\ :\ 2\le i\le j\big)\subset\C^{m-1}_x\times\C^j_{y_1,\dots,y_j}.
\]
For a partition $P=(p_1,\dots,p_r)$ of $j$, $D^j(f,P)$ is the subscheme of $\C^{m-1}\times\C^r$ defined by the same divided differences with $y$'s repeated according to $P$ (confluent divided differences. Equivalently, by the ideal $I_{D^j}$ \emph{together with the diagonal equations} of $\Delta_P$, i.e.\ its image in the coordinate ring of the diagonal). The symmetric group $S_j$ acts on $D^j(f)$ by permuting the $y_i$. Write $e(j,P)=m+1-2j+r$ for the expected dimension.

\begin{enumerate}[label=\textup{(MM)}]
\item  \textbf{Marar--Mond.} Let $f$ be $\cA$-finite (in particular $\Ae$-finite) of corank one. If $e(j,P)\ge0$ then $D^j(f,P)$ is either \emph{empty} or an isolated complete intersection singularity (ICIS) of dimension $e(j,P)$. If $e(j,P)<0$ then $D^j(f,P)\subseteq\{0\}$ set-theoretically (it may be a fat point). References: \cite[Thm~2.14(ii) and Cor.~2.15]{MM89}. The partitioned statement is also quoted in this form in \cite[Def.~2.3, Cor.~2.5]{NBPT}, and the unpartitioned one in \cite[Prop.~3.2, Thm~3.3]{NBPS}.
\end{enumerate}
In particular, $D^2(f)$ is an ICIS of dimension $m-1$ (nonempty, since it contains the origin and $\Sigma$ below). The scheme $\Sigma:=D^2(f,(2))=V(p_y,q_y)\subset\C^m$ is an ICIS of dimension $m-2$: for $m=3$ it is a nonempty reduced curve, while for $m=2$ it is a fat point at $0$. Finally, $D^3(f)$ is empty or an ICIS of dimension $m-2$, while $D^{m+1}(f)$ and $D^m(f,(2,1^{m-2}))$ are supported at the origin. An ICIS of positive dimension is reduced and smooth off the origin. Here ``off the origin'' may be taken globally on the affine cone by $\C^*$-invariance.  \emph{The triple point scheme $D^3(f)$ may be empty} even for $m=3$. See Lemma~\ref{lem:Lambda3} and Convention~\ref{conv:empty}. Every statement below about ``the triple curve'' allows the empty set.

The set-theoretic meaning: for $y_1\ne y_2$, $(x,y_1,y_2)\in D^2(f)$ iff $f(x,y_1)=f(x,y_2)$ (a double point pair); $(x,y,y)\in D^2(f)$ iff $p_y=q_y=0$ at $(x,y)$, i.e.\ $f$ is not immersive there (the cross-cap locus $\Sigma$).

The zero-dimensional multiple-point schemes above are cones, and hence are supported at the vertex. For $m=3$, $D^4(f)$ and $D^3(f,(2,1))$ exclude $A_0^4$ and $S_0A_0$ away from the origin. Thus only $A_0$, $A_0^2$, $A_0^3$ and $S_0$ occur on $X\setminus\{0\}$. For $m=2$, $D^3(f)$ and $D^2(f,(2))$ exclude triple points and cross-caps there, leaving only $A_0$ and $A_0^2$.

\subsection{\texorpdfstring{$\C^*$}{C*}-orbits and Euler fields}\label{sec:orbits}

\begin{lemma}\label{lem:orbit}
Let $C\subset\C^N$ be a $\C^*$-invariant (for positive weights) reduced curve and $P\in C\setminus\{0\}$ a smooth point of $C$. Then $T_PC=\C\cdot E(P)$ with $E(P)\ne0$.
\end{lemma}

\begin{proof}
Each irreducible component of $C$ is $\C^*$-invariant (connectedness of $\C^*$) and $1$-dimensional, hence the closure of the orbit of any of its points $\ne0$. The orbit map $\C^*\to C$, $\lambda\mapsto\lambda\cdot P$, has finite fibres (the stabiliser of $P\ne0$ is finite for positive weights), hence is \'etale onto its image at $\lambda=1$. Its differential there is $E(P)$, so $E(P)\ne0$ spans $T_PC$.
\end{proof}

\section{The image, the double-point function and the triple-point function}\label{sec:image}

\subsection{\texorpdfstring{$\cO_X\subset\cO_S$}{OX-OS} and the equation of \texorpdfstring{$X$}{X}}\label{sec:equation}

Since $p(0,y)=cy^\mu$ with $c$ a nonzero scalar (positivity of the weights), $\cO_S=\C[x,y]$ is a free $\C[x,Y]$-module ($Y\mapsto p$) with basis $1,y,\dots,y^{\mu-1}$. In particular $f$ is finite. Let $Q$ be the $\mu\times\mu$ matrix of multiplication by $q$ in this basis and $G:=\det(Z\cdot\Id-Q)\in\cO_T$, weighted homogeneous of degree $\mu k$. Then $G\circ f=0$ (Cayley--Hamilton: $q$ is a root of its own characteristic polynomial), so $(G)\subseteq\ker f^*$, and $\ker f^*$ is a height-one prime (the image of $\Spec\cO_S\to\Spec\cO_T$ under a finite map is an irreducible hypersurface). Since $D^2(f)$ has dimension $m-1<m$ (MM), $f$ is generically one-to-one, so the degree of $X\to\C^m_{x,Y}$ is $\mu=\deg_ZG$. Hence the characteristic polynomial $G$ equals the minimal polynomial of $q$ over $\C(x,Y)$, $\ker f^*=(G)$, and $G$ is reduced, and we have
\[
\cO_X=f^*\cO_T=\cO_T/(G),\qquad\Hilb\cO_X=(1-t^{\mu k})\Hilb\cO_T .
\]
Since $\cO_X$ is a hypersurface ring, it is Cohen--Macaulay of dimension $m$, and a domain. We have that $\cO_X\subseteq\cO_S$ is a finite birational extension with $\cO_S$ regular: {\em $\cO_S$ is the normalisation of $\cO_X$.} In particular Seidenberg's theorem (Se) applies to $\cO_X\subset\cO_S$.

\subsection{The double-point function \texorpdfstring{$\Lambda$}{Lambda}}\label{sec:double}

Let $p[12],q[12]\in\C[x,y_1,y_2]$ be the first divided differences. As a polynomial in $y_2$, $p[12]=c\,y_2^{\mu-1}+\dots$ has constant leading coefficient. Define
\[
\Lambda:=\Res_{y_2}\big(p[12],q[12]\big)\in\C[x,y_1]=\cO_S .
\]

\begin{lemma}\label{lem:Lambda}
\begin{enumerate}[label=\textup{(\alph*)}]
\item $\Lambda$ is weighted homogeneous of degree $(\mu-1)(k-v)$.
\item Over the splitting field, $\Lambda=c'\prod_{j=2}^{\mu}q[1j]$, where $y_2,\dots,y_\mu$ are the roots of $p(x,\cdot)=p(x,y_1)$ other than $y_1$ (with multiplicity) and $q[1j]=(q(x,y_1)-q(x,y_j))/(y_1-y_j)$.
\item $V(\Lambda)=D:=\overline{\{s\in S:\ \#f^{-1}f(s)\ge2\}}=\pi_1(D^2(f))$, where $\pi_1(x,y_1,y_2)=(x,y_1)$; $D$ has pure dimension $m-1$ and $D\supseteq\Sigma$.
\end{enumerate}
\end{lemma}

\begin{proof}
(a),(b) For $A$ of degree $\mu-1$ in $y_2$ with constant leading coefficient $c$, $\Res_{y_2}(A,B)=c^{\deg_{y_2}B}\prod_{A(r)=0}B(r)$. The roots of $p[12](x,y_1,\cdot)$ are the $y_j$, $j\ge2$, because $p(x,y_2)-p(x,y_1)=(y_2-y_1)p[12]$. Each factor $q[12](x,y_1,y_j)$ is weighted homogeneous of degree $k-v$ in $(x,y_1)$ once $y_j$ (of weight $v$) is expressed through $(x,y_1)$.

(c) $\Lambda(x,y_1)=0$ if and only if $p[12](x,y_1,\cdot)$ and $q[12](x,y_1,\cdot)$ have a common root $y_2$ (the leading coefficient of $p[12]$ in $y_2$ never vanishes), that is, if and only if $(x,y_1,y_2)\in D^2(f)$ for some $y_2$. For $y_2\ne y_1$ this is a double point pair and for $y_2=y_1$ it is a point of $\Sigma$. Hence $V(\Lambda)=\pi_1(D^2(f))$, which is the closure of the set of genuine double points together with $\Sigma$. Since $D^2(f)$ has pure dimension $m-1$ (MM) and $\pi_1$ is finite, $\pi_1(D^2)$ has pure dimension $m-1$; $\Sigma$ has dimension $m-2$ and is contained in the closure of the double points (every component of $D^2(f)$ has dimension $m-1$ and the diagonal part $D^2\cap\{y_1=y_2\}\cong\Sigma$ has dimension $m-2$, so the diagonal points are limits of off-diagonal points).
\end{proof}

\begin{lemma}[$\Lambda$ is reduced]\label{lem:Lambda-reduced}
$\Lambda$ is square-free, and $V(\Lambda)=D$ is a reduced hypersurface. Hence $(\Lambda)=I(D)$ is the reduced ideal of $D$.
\end{lemma}

\begin{proof}
Let $A=\C[x,y_1]=\cO_S$ and $B=\C[x,y_1,y_2]/(p[12])$. Since $p[12]$ has $y_2$-degree $\mu-1$ with constant leading coefficient, $B$ is a free $A$-module with basis $1,y_2,\dots,y_2^{\mu-2}$, and the resultant is the determinant of multiplication by $q[12]$ on this free module:
\[
\Lambda=\Res_{y_2}(p[12],q[12])=c^{\deg_{y_2}q[12]}\det\big(q[12]\cdot\,:B\to B\big)
\]
(evaluate at the companion matrix). For a height-one prime $\fp$ of the regular ring $A$ with discrete valuation ring $A_\fp$, an injective $A_\fp$-linear endomorphism $\varphi$ of a free module of finite rank satisfies $\ord_\fp(\det\varphi)=\length_{A_\fp}(\coker\varphi)$ (Smith normal form). Here $\coker(q[12]\cdot)=B/q[12]B=\C[x,y_1,y_2]/(p[12],q[12])=\cO_{D^2(f)}$, and multiplication by $q[12]$ is injective on $B\otimes A_\fp$ for every $\fp$ (otherwise $\Lambda=0$, but $V(\Lambda)=D\ne S$). Hence, for the prime $\fp_Z$ of an irreducible component $Z$ of $D$,
\[
\ord_Z(\Lambda)=\length_{A_{\fp_Z}}\big(\cO_{D^2(f)}\otimes_AA_{\fp_Z}\big)=\sum_{\widetilde Z\to Z}[\kappa(\widetilde Z):\kappa(Z)]\cdot\length\cO_{D^2(f),\widetilde Z},
\]
the sum over the irreducible components $\widetilde Z$ of $D^2(f)$ dominating $Z$ (all components of $D^2(f)$ have dimension $m-1$ and $\pi_1$ is finite, so each maps onto a component of $D$, and $D=\pi_1(D^2)$). $D^2(f)$ is reduced (MM), so each local ring at a generic point $\widetilde Z$ is a field and the lengths are $1$. Finally, by (MG)/(St), a generic point $s$ of $Z$ is an $A_0^2$ point (the points of $Z$ of type $A_0^{j\ge3}$ or $S_0$ form a subset of dimension $\le m-2$, because their images are the triple-point and cross-cap loci of dimension $\le m-2$ and $f$ is finite); over such $s$ the fibre of $\pi_1:D^2(f)\to D$ is the single reduced point $(s,s')$, $s'$ the unique partner. So exactly one component $\widetilde Z$ dominates $Z$, with $[\kappa(\widetilde Z):\kappa(Z)]=1$, and $\ord_Z(\Lambda)=1$. As $V(\Lambda)=D$ (Lemma~\ref{lem:Lambda}(c)) and $A$ is a UFD, $\Lambda$ is square-free and $(\Lambda)=I(D)$.
\end{proof}

\subsection{The ordered double-point surface is a normal domain (\texorpdfstring{$m=3$}{m3})}\label{sec:D2normal}

\begin{lemma}\label{lem:D2normal}
Let $m=3$. The ring $\cO_{D^2}=\C[x_1,x_2,y_1,y_2]/(p[12],q[12])$ is a complete intersection of dimension $2$, smooth off the origin, normal, and a \emph{domain}. Its ring of invariants $\cO_{D^2}^{S_2}$ under $y_1\leftrightarrow y_2$ is also a normal domain. Its complete-intersection structure and smoothness off the origin are established in Lemma~\ref{lem:normalisation}(a).
\end{lemma}

\begin{proof}
$D^2(f)$ is an ICIS surface (MM), so $\cO_{D^2}$ is CM and its localisations at points $\ne0$ are regular. Serre's criterion ($R_1$ from the isolated singularity, $S_2$ from CM) gives normality. Normality alone does not give irreducibility, but the grading does: every irreducible component of $D^2$ is invariant under the connected group $\C^*$ acting with positive weights, hence contains the vertex $0$. The local ring of $D^2$ at $0$ is normal, hence a domain, so only one component passes through $0$. Therefore $D^2$ is irreducible, and being reduced, $\cO_{D^2}$ is a domain. The ring of invariants of a normal domain under a finite group is a normal domain (it is a subring of a domain, and integrally closed in its fraction field $\C(D^2)^{S_2}$).
\end{proof}

\subsection{The triple-point function \texorpdfstring{$\Lambda_3$}{Lambda3} (\texorpdfstring{$m=3$}{m3})}\label{sec:triple}

Let $p[123],q[123]$ be the second divided differences (symmetric in $y_1,y_2,y_3$). For $\mu\ge3$, $\deg_{y_3}p[123]=\mu-2$ with constant leading coefficient, and we define
\[
\Lambda_3:=\Res_{y_3}\big(p[123],q[123]\big)\in\C[x,y_1,y_2]^{S_2},\qquad\deg\Lambda_3=(\mu-2)(k-2v),\qquad \Lambda_3=c''\prod_{j\ge3}q[12j]
\]
over the splitting field, with $y_3,\dots,y_\mu$ the roots of $p(x,\cdot)=p(x,y_1)$ other than $y_1,y_2$. On $D^2(f)$ (where $p(x,y_2)=p(x,y_1)$) these are the remaining preimages. $\Lambda_3$ defines a function on $D^2(f)$ and, being symmetric, a function $\bL$ on $\tD=D^2/S_2$ (\S\ref{sec:normalisation}).

\begin{convention}[empty triple locus]\label{conv:empty}
If $\mu=2$ then $p[123]$ is the nonzero constant $c$, $D^3(f)=\emptyset$, and we \emph{define} $\Lambda_3:=1$ (the determinant of the endomorphism of the zero module, in the norm interpretation of the proof below. This avoids implementation-dependent conventions for resultants involving a constant). If $\mu\ge3$ and $V(\Lambda_3)\cap D^2(f)\subseteq\{0\}$, then $\bL$ is a homogeneous function on $\tD$ vanishing at most at the vertex, hence (positive grading) a nonzero \emph{constant}, and $D^3(f)=\emptyset$. In both cases $\deg\Lambda_3=0$, $(\bL)$ is the unit ideal, and the ``triple curve'' $\tT$ below is empty. \emph{Example}: $f=(x_1,x_2,y^4+x_1y,y^2)$ with weights $(a_1,a_2,v)=(3,1,1)$ and degrees $(3,1,4,2)$; subtracting the square of the last coordinate from the third gives a cross-cap times a line, so $f$ is stable, $\mu=4$ in the normalised form (Remark~\ref{rem:normalisation}: here $p=y^4+x_1y$, $q=y^2$) and $q[123]=1$, $D^3=\emptyset$, $\Lambda_3=1$. (Here $q(0,y)=y^2\ne0$ as well, so the exchanged normal form, with $(\mu',k')=(2,4)$, is also available; cf.\ Remark~\ref{rem:symmetry}.) For $\mu\ge3$, $\Ae$-finiteness forces $k\ge2v$: if $k<2v$ then $q[123]$, of degree $k-2v<0$, vanishes identically and
\[
\cO_{D^3(f)}=\cO_{D^2(f)}[y_3]/(p[123]).
\]
Since $p[123]$ has $y_3$-degree $\mu-2$ and nonzero constant leading coefficient, this is a free module of rank $\mu-2$ over $\cO_{D^2(f)}$. Thus $D^3(f)\to D^2(f)$ is finite and surjective, giving $\dim D^3(f)=2$, contrary to (MM). If $k=2v$ then $q[123]$ is a constant, which is nonzero for the same reason, and we are in the empty case. Conversely, if $k>2v$ then $\deg\bL=(\mu-2)(k-2v)>0$, and a homogeneous function of positive degree on the two-dimensional cone $\tD$ vanishes on a nonempty curve, so the triple curve is nonempty. Hence for $\mu\ge3$ the empty case occurs exactly when $k=2v$, i.e.\ exactly when $\deg\Lambda_3=0$.
\end{convention}

\begin{lemma}[$\Lambda_3$ is a reduced equation of the triple curve]\label{lem:Lambda3}
Let $m=3$. Off the origin, $V(\Lambda_3)\cap D^2(f)$ is exactly the set $\tT$ of pairs $(x,y_1,y_2)\in D^2$ admitting a third preimage $y_3\notin\{y_1,y_2\}$ (the preimage of the triple-point curve $T_X\subset X$); $\tT$ is a reduced curve or empty, and $\Lambda_3$ vanishes to first order along each of its components: $(\Lambda_3)=I(\tT)$ in the normal domain $\cO_{D^2}$, and $(\bL)=I(\tT_X)$ in $\cO_{\tD}$, where $\tT_X=\tT/S_2$. In the empty case both ideals are the unit ideal.
\end{lemma}

\begin{proof}
The empty case is covered by Convention~\ref{conv:empty}. Assume $\mu\ge3$ and $\tT\ne\emptyset$. On $D^2\setminus\{0\}$, $\Lambda_3(x,y_1,y_2)=0$ iff $p[123](x,y_1,y_2,\cdot)$ and $q[123](x,y_1,y_2,\cdot)$ have a common root $y_3$. If $y_3\notin\{y_1,y_2\}$ then $(x,y_1,y_2,y_3)\in D^3(f)$ is a triple point. If $y_3=y_1$ (say) then $(x,y_1,y_1,y_2)\in D^3(f,(2,1))$, which has expected dimension $0$ (MM), hence is at the origin. So off the origin $V(\Lambda_3)\cap D^2=\pi_{12}(D^3)=\tT$, a curve (image of the ICIS curve $D^3$ under the finite map $\pi_{12}$). For the order of vanishing we repeat the norm argument of Lemma~\ref{lem:Lambda-reduced} one level up: $B'=\cO_{D^2}[y_3]/(p[123])$ is free of rank $\mu-2$ over $A'=\cO_{D^2}$ (constant leading coefficient), $\Lambda_3|_{D^2}=c\det(q[123]\cdot:B'\to B')$, and $B'/q[123]B'=\cO_{D^3(f)}$. $\cO_{D^2}$ is a normal domain by Lemma~\ref{lem:D2normal}, which uses only (MM) and the grading, so there is no circularity. For the height-one prime of a component $\widetilde Z$ of $\tT$ the local ring is a DVR and $\ord_{\widetilde Z}(\Lambda_3)=\sum_{W\to\widetilde Z}[\kappa(W):\kappa(\widetilde Z)]\length\cO_{D^3,W}$ over the components $W$ of the reduced curve $D^3$ dominating $\widetilde Z$. At a generic point of $\widetilde Z$ the multigerm is of type $A_0^3$ (the $A_0^4$ and $S_0A_0$ points lie at the origin by (St), and $\tT$ has no component inside the fixed curve $\{y_1=y_2\}$, since $\tT\cap\{y_1=y_2\}$ lies over $D^3(f,(2,1))$, which is at the origin), so the third preimage $y_3$ is unique: exactly one component $W$ dominates $\widetilde Z$, birationally, and $\ord_{\widetilde Z}(\Lambda_3)=1$. In the normal domain $\cO_{D^2}$ the principal ideal $(\Lambda_3)$ is unmixed of height one \ref{D4}, hence equals the intersection of the height-one primes of the components of $\tT$: $(\Lambda_3)=I(\tT)$ is reduced. Passing to $S_2$-invariants: $D^2\to\tD$ is an unramified double cover near $\tT\setminus0$ (the fixed curve meets $\tT$ only at $0$, as just shown), $\bL$ is the invariant function with pullback $\Lambda_3$, so $\ord=1$ along each component of $\tT_X$, and the same unmixedness argument in the normal domain $\cO_{\tD}$ (Lemma~\ref{lem:D2normal}) gives $(\bL)=I(\tT_X)$.
\end{proof}

\section{Local models}\label{sec:local}

\subsection{What has to be computed locally}\label{sec:local0}

All modules in this paper are canonically attached to the germ $f$ (or to its image $X$ with its normalisation $f:S\to X$) and are preserved by $\cA$-equivalence of multigerms: $L=\{\eta\in\theta_S:\eta(\cO_X)\subseteq\cO_X\}$, the conductor $\fc=\Ann_{\cO_X}(\cO_S/\cO_X)$, $\cO_{D_X}=\cO_X/\fc$, $\Der(\cO_{D_X})$, the cross-cap curve $C_X=f(\Sigma)$ and its reduced ideal, the normalisation $\cO_{\tD}$ of $\cO_{D_X}$ and its conductor. By Lemma~\ref{lem:comparison} and (MG)/(St), any statement ``$M=N$ for two such modules'' holds globally on the punctured cone as soon as it holds in the analytic local rings at the points of $X\setminus\{0\}$, where the multigerm is $\cA$-equivalent to one of the normal forms $A_0^j$ ($j\le m$), $S_0$. This section computes everything needed in those normal forms.

Notation: $w$ denotes the coordinates transverse to the ones displayed (the trivial factor $\C^{r}$); all rings are analytic local rings at the point in question, but since the normal forms are polynomial and the computations are of kernels/annihilators, they may equally be done in polynomial rings localised at the point.

\subsection{Normal crossings: \texorpdfstring{$X=\{u_1u_2\cdots u_j=0\}\subset\C^{n+1}$, $j\le n+1$}{ }}\label{sec:nc}

Let $\cO=\C\{u_1,\dots,u_j,w\}$ ($w=(w_1,\dots,w_{n+1-j})$), $X=V(u_1\cdots u_j)$, $H_i=V(u_i)$, $\cO_X=\cO/(u_1\cdots u_j)$. The normalisation of $X$ is $\widetilde X=\bigsqcup_iH_i$, $\cO_{\widetilde X}=\prod_i\cO_{H_i}$, $\cO_{H_i}=\cO/(u_i)$, and this is the multigerm $f$ at an $A_0^j$ point (the $j$ branches are the inclusions $H_i\hookrightarrow\C^{n+1}$).

\begin{proposition}\label{prop:nc}
\begin{enumerate}[label=\textup{(\alph*)}]
\item \textup{(Seminormality.)} $\cO_X=\big\{(h_1,\dots,h_j)\in\prod_i\cO_{H_i}:\ h_i|_{H_i\cap H_l}=h_l|_{H_i\cap H_l}$ for all $i<l\big\}$.
\item \textup{(Conductor.)} The conductor of $\cO_X\subset\cO_{\widetilde X}$ is $\fc=\prod_i\fc_i$ with $\fc_i=\big(\prod_{l\ne i}u_l\big)\cO_{H_i}$, the reduced ideal in $H_i$ of $H_i\cap\bigcup_{l\ne i}H_l$. Thus $\fc$ is the reduced ideal of the double locus $\widetilde D=\bigsqcup_i(H_i\cap\bigcup_{l\ne i}H_l)$ of the normalisation, and $\cO_X/\fc=\cO_{D_X}$ with $D_X=\bigcup_{i<l}(H_i\cap H_l)$ \emph{reduced}.
\item \textup{(Logarithmic fields.)} $\Dlog X$ is free on $u_1\partial_{u_1},\dots,u_j\partial_{u_j},\partial_{w_1},\dots,\partial_{w_{n+1-j}}$.
\item \textup{($j=2$.)} $D_X=V(u_1,u_2)$ is smooth, $\Der(\cO_{D_X})=\bigoplus_r\cO_{D_X}\partial_{w_r}$, and the restriction map $\Dlog X\to\Der(\cO_{D_X})$ is surjective.
\item \textup{($j=3$.)} $\cO_{D_X}=\cO/(u_1u_2,u_1u_3,u_2u_3)$, the ring of three coordinate planes $\Pi_i=V(u_l,u_{l'})$ ($\{i,l,l'\}=\{1,2,3\}$) in $\C^3_{u}\times\C^{n-2}_w$ meeting along $T=V(u_1,u_2,u_3)$. Its normalisation is $\bigsqcup_i\Pi_i$, $\cO_{\Pi_i}=\C\{u_i,w\}$, with conductor $\prod_i(u_i)\cO_{\Pi_i}$, the reduced ideal of $T$ in each plane. Moreover
\begin{multline*}
\Der(\cO_{D_X})=\Big\{\delta:\ \delta(u_i)\in u_i\,\C\{u_i,w\}\ (i=1,2,3),\ \delta(w_r)\in\cO_{D_X}\Big\}\\
=\bigoplus_i u_i\C\{u_i,w\}\partial_{u_i}\ \oplus\ \bigoplus_r\cO_{D_X}\partial_{w_r},
\end{multline*}
and the restriction map $\Dlog X\to\Der(\cO_{D_X})$ is surjective.
\end{enumerate}
\end{proposition}

\begin{proof}
(a) The sequence $0\to\cO_X\to\prod_i\cO_{H_i}\to\prod_{i<l}\cO_{H_i\cap H_l}$ (differences of restrictions) is exact: injectivity because $X$ is reduced and the $H_i$ cover it. For exactness in the middle, given $(h_i)$ agreeing on pairwise intersections, lift $h_1$ to $\tilde h_1\in\cO$. Then $h_2-\tilde h_1|_{H_2}$ vanishes on $H_1\cap H_2$, so equals $u_1g_2$ on $H_2$; put $\tilde h_2=\tilde h_1+u_1g_2$, which restricts to $h_1$ on $H_1$ and to $h_2$ on $H_2$. Then $h_3-\tilde h_2|_{H_3}$ vanishes on $H_3\cap(H_1\cup H_2)=V(u_1u_2)\cap H_3$, whose reduced ideal in $\cO_{H_3}$ is $(u_1u_2)$ (a normal crossing in $H_3$), so $h_3-\tilde h_2|_{H_3}=u_1u_2g_3$ and $\tilde h_3=\tilde h_2+u_1u_2g_3$ works; continue.

(b) $(h_i)\in\fc$ iff $(h_ig_i)\in\cO_X$ for all $(g_i)$; taking $g_l=\delta_{il}$ gives $h_i|_{H_i\cap H_l}=0$ for all $l\ne i$, i.e.\ $h_i\in(\prod_{l\ne i}u_l)\cO_{H_i}$ (reduced ideal of a normal crossing in $H_i$). Conversely such $(h_i)$ satisfy $h_ig_i|_{H_i\cap H_l}=0=h_lg_l|_{H_i\cap H_l}$. The quotient $\prod\cO_{H_i}/\fc$ is $\prod_i\cO_{H_i}/(\prod_{l\ne i}u_l)$, reduced; $\cO_X/\fc\subseteq$ it is reduced, and its zero set is $\bigcup_{i<l}H_i\cap H_l$.

(c) $\xi=\sum A_i\partial_{u_i}+\sum B_r\partial_{w_r}$ satisfies $\xi(u_1\cdots u_j)\in(u_1\cdots u_j)$ iff $\sum_iA_i\prod_{l\ne i}u_l\in(u_1\cdots u_j)$ iff $u_i\mid A_i$ for each $i$ (reduce modulo $u_i$: the only surviving term is $A_i\prod_{l\ne i}u_l$, and $\prod_{l\ne i}u_l$ is a nonzerodivisor mod $u_i$). (This is also an instance of (S): the determinant of the displayed fields is $u_1\cdots u_j$.)

(d) Clear from (c): $\partial_{w_r}$ restrict to the coordinate fields of the smooth $D_X$.

(e) The three planes $\Pi_i=V(u_l,u_{l'})$ meet pairwise exactly along $T=V(u_1,u_2,u_3)$, and $D_X=\Pi_1\cup\Pi_2\cup\Pi_3$ is a normal crossing of planes in $\C^3_u\times\C^{n-2}_w$. The same argument as in (a),(b) (with the pairwise intersections all equal to $T$) gives $\cO_{D_X}=\{(h_i)\in\prod_i\C\{u_i,w\}:h_1|_T=h_2|_T=h_3|_T\}$ and conductor $\prod_i(u_i)\C\{u_i,w\}$, the reduced ideal of $T$ in each plane. Derivations: $\delta\in\Der(\cO_{D_X})$ is determined by $\delta(u_1),\delta(u_2),\delta(u_3),\delta(w_r)\in\cO_{D_X}\subseteq\prod_i\C\{u_i,w\}$ subject to $\delta(u_iu_l)=u_i\delta(u_l)+u_l\delta(u_i)=0$ for $i\ne l$. Restrict this relation to the plane $\Pi_i$, on which $u_l=0$ and $u_i$ is a coordinate: it reads $u_i\cdot\delta(u_l)|_{\Pi_i}=0$, so $\delta(u_l)|_{\Pi_i}=0$ for all $l\ne i$. Hence $\delta(u_l)$ is supported on $\Pi_l$: $\delta(u_l)=(0,\dots,g_l,\dots,0)$ with $g_l\in\C\{u_l,w\}$; and since the components of an element of $\cO_{D_X}$ agree on $T$, $g_l|_T=0$, i.e.\ $g_l\in u_l\C\{u_l,w\}$. $\delta(w_r)\in\cO_{D_X}$ is unconstrained. Conversely any such data respect the relations $u_iu_l=0$ and define a derivation. This is the displayed description. Surjectivity from $\Dlog X$: the fields $u_i\partial_{u_i}$, $\partial_{w_r}$ of (c) induce the derivations with $\delta(u_i)=(0,\dots,u_i,\dots,0)$, $\delta(w_r)=1$; given $\delta$ with $g_l=u_l\gamma_l(u_l,w)$, the function $\Gamma_l:=\gamma_l(u_l,w)\in\cO$ restricts to an element of $\cO_{D_X}$ and $\sum_l\Gamma_l\cdot(u_l\partial_{u_l})+\sum_r\delta(w_r)\partial_{w_r}$ induces $\delta$ (on $\Pi_i$, $i\ne l$, both $\Gamma_lu_l$ and $g_l$ vanish).
\end{proof}

\subsection{The cross-cap}\label{sec:crosscap}

We use the model $S_0$ given by $f_0(x_1,x_2,y)=(x_1,y^2,x_1y,x_2)$.

(For $m=2$ drop $x_2$; but recall that for $m=2$ cross-caps occur only at the origin, so this model is needed for $m=3$ only.) Target coordinates $(X_1,Y,Z,w)$ with $w=x_2$; $X=V(G)$, $G=Z^2-X_1^2Y$ (the Whitney umbrella times $\C_w$).

\begin{proposition}\label{prop:crosscap}
\begin{enumerate}[label=\textup{(\alph*)}]
\item $\cO_X=\C\{x_1,y^2,x_1y,w\}=\{h\in\C\{x_1,y,w\}:\ h(0,y,w)=h(0,-y,w)\}$, i.e.\ $h\in\cO_X$ iff $h(s)=h(s')$ for all double point pairs $s=(0,y,w)$, $s'=(0,-y,w)$. Equivalently $h[12]\in(y_1+y_2,\,x_1)=I_{D^2}$.
\item The conductor of $\cO_X\subset\cO_S$ is $\fc=x_1\cO_S$, the reduced ideal of the double locus $D=V(x_1)$ of the source; $\Lambda=q[12]=x_1$ (up to a unit). As an ideal of $\cO_X$, $\fc=(X_1,Z)\cO_X$ (\emph{not} $(X_1)\cO_X$: $\cO_X/(X_1)=\C\{Y,Z,w\}/(Z^2)$ contains the nonzero nilpotent $Z$). Hence
\[
\cO_{D_X}=\cO_X/\fc=\C\{X_1,Y,Z,w\}/(G,X_1,Z)=\C\{Y,w\},
\]
so $D_X=V(X_1,Z)$ is a smooth plane, and the cross-cap curve is $C_X=V(X_1,Y,Z)=\{Y=0\}\subset D_X$.
\item $\Dlog X$ is generated by
\[
\xi_1=X_1\partial_{X_1}+Z\partial_Z,\quad \xi_2=-X_1\partial_{X_1}+2Y\partial_Y,\quad \xi_3=2Z\partial_Y+X_1^2\partial_Z,\quad \xi_4=Z\partial_{X_1}+X_1Y\partial_Z,\quad \partial_w ,
\]
where $\xi_1(G)=2G$ and $\xi_2(G)=\xi_3(G)=\xi_4(G)=\partial_w(G)=0$. (The field $\xi_1$ is \emph{not} a multiple of the weighted Euler field $X_1\partial_{X_1}+2Y\partial_Y+2Z\partial_Z$. Only the relation $\xi_1(G)=2G$ is used.)
\item The image of $\Dlog X\to\Der(\cO_{D_X})=\cO_{D_X}\partial_Y\oplus\cO_{D_X}\partial_w$ is $\cO_{D_X}\,Y\partial_Y\oplus\cO_{D_X}\partial_w=\Der(\cO_{D_X};C_X)$, the derivations of $D_X$ preserving the ideal $(Y)$ of $C_X$. In particular every logarithmic field is tangent to $C_X$, and every derivation of $\cO_{D_X}$ tangent to $C_X$ lifts to a logarithmic field.
\item $D^2(f_0)=V(y_1+y_2,x_1)\cong\C^2_{y_1,w}$ is smooth, $S_2$ acts by $y_1\mapsto-y_1$, $D^2/S_2\cong\C^2_{y_1^2,w}\cong D_X$ (via $Y=y_1^2$), so $\tD\to D_X$ is an isomorphism and the branch curve of $D^2\to D^2/S_2$ is $C_X$.
\end{enumerate}
\end{proposition}

\begin{proof}
(a) $\C\{x_1,y^2,x_1y,w\}$ is spanned by the monomials $x_1^ay^bw^c$ with $b$ even or $a\ge1$ (for $a\ge1$, $b$ odd: $x_1^ay^b=(x_1y)x_1^{a-1}y^{b-1}$). These are exactly the monomials whose coefficient in $h$ is unconstrained by $h(0,y,w)=h(0,-y,w)$, which kills the coefficients of $y^{\text{odd}}w^c$. The divided-difference form: $h[12]$ restricted to $V(y_1+y_2,x_1)$ is $(h(0,y_1,w)-h(0,-y_1,w))/(2y_1)$.

(b) $h\in\fc$ iff $hg\in\cO_X$ for all $g$; with $g=1$ and $g=y$: $h(0,y,w)$ and $yh(0,y,w)$ both even in $y$ forces $h(0,y,w)=0$, i.e.\ $x_1\mid h$. Conversely $x_1g\in\cO_X$ for all $g$ since $x_1y^b\in\cO_X$. The double points are the pairs $(0,\pm y,w)$, so $D=V(x_1)$; $q[12]=(x_1y_1-x_1y_2)/(y_1-y_2)=x_1$. In target coordinates $x_1\cO_S=x_1\C\{x_1,y,w\}$ is generated over $\cO_X$ by $x_1=X_1$ and $x_1y=Z$ (every $x_1^ay^bw^c$ with $a\ge1$ is $X_1\cdot(\ldots)$ if $b$ is even and $Z\cdot(\ldots)$ if $b$ is odd), so $\fc=(X_1,Z)\cO_X$, and $\cO_X/\fc=\C\{X_1,Y,Z,w\}/(Z^2-X_1^2Y,X_1,Z)=\C\{Y,w\}$.

(c) $G$ is weighted homogeneous for $X_1,Y,Z$ of weights $1,2,2$, and $\xi_1(G)=2Z^2-2X_1^2Y=2G$. Hence any $\xi$ with $\xi(G)=hG$ differs from $\tfrac h2\xi_1$ by a field annihilating $G$, i.e.\ a syzygy of $(G_{X_1},G_Y,G_Z)=(-2X_1Y,-X_1^2,2Z)$. Since $Z$ is a nonzerodivisor modulo $(X_1Y,X_1^2)$, the syzygies of $(X_1Y,X_1^2,Z)$ are generated by the Koszul syzygies involving $Z$ and the syzygies of $(X_1Y,X_1^2)$, which are generated by $(X_1,-Y)$: this gives $\xi_4\propto(Z,0,X_1Y)$, $\xi_3\propto(0,2Z,X_1^2)$ and $\xi_2\propto(-X_1,2Y,0)$ (check: $\xi_2(G)=2X_1^2Y-2X_1^2Y=0$, $\xi_3(G)=-2X_1^2Z+2X_1^2Z=0$, $\xi_4(G)=-2X_1YZ+2X_1YZ=0$). $\partial_w$ is tangent.

(d) Restrict to $D_X=V(X_1,Z)$: $\xi_1,\xi_3,\xi_4\mapsto0$, $\xi_2\mapsto2Y\partial_Y$, $\partial_w\mapsto\partial_w$. A derivation $\delta=A\partial_Y+B\partial_w$ of $\C\{Y,w\}$ preserves $(Y)$ iff $Y\mid A$.

(e) $p=y^2$: $p[12]=y_1+y_2$; $q=x_1y$: $q[12]=x_1$.
\end{proof}

\begin{remark}\label{rem:S0A0}
The $S_0A_0$ multigerm (cross-cap plus a transverse branch) and the $A_0^{m+1}$ multigerm occur only at the origin (\S\ref{sec:stable}) and need no local computation.
\end{remark}

\subsection{Summary of the local structure away from the origin}\label{sec:localsummary}

\begin{center}
\footnotesize
\setlength{\tabcolsep}{4pt}
\begin{tabular}{@{}lllllll@{}}
\toprule
type of $P$ & $X$ at $P$ & $D_X$ at $P$ & $\tD\to D_X$ & $C_X$ & $\Lambda$ & $\bL$ \\
\midrule
$A_0$ & smooth & $\emptyset$ & --- & --- & unit & unit \\
$A_0^2$ & $u_1u_2=0$ & smooth $V(u_1,u_2)$ & iso & absent & red.\ eq.\ of $D$ & unit \\
$A_0^3$ ($m=3$) & $u_1u_2u_3=0$ & 3 planes along $T$ & $\bigsqcup\Pi_i\to D_X$ & absent & red.\ eq.\ of $D$ & red.\ eq.\ of $T$ on each $\Pi_i$ \\
$S_0$ ($m=3$) & umbrella$\times\C$ & smooth plane & iso & line $Y{=}0$ & $x_1$ & unit \\
\bottomrule
\end{tabular}
\end{center}

(The entries for $\Lambda,\Lambda_3$ are Lemmas~\ref{lem:Lambda-reduced}, \ref{lem:Lambda3} together with Prop.~\ref{prop:nc}(b),(e), \ref{prop:crosscap}(b): in each case the conductor is the reduced ideal of the double locus, which is $V(\Lambda)$, resp.\ of the triple locus, which is $V(\Lambda_3)$.)

\section{Image derivations and the conductor}\label{sec:liftable}

\subsection{Source fields admitting target lifts}\label{sec:liftmodule}

Recall $L=\{\eta\in\theta_S:\ tf(\eta)\in\omega f(\Theta_T)\}$ and the target module $\operatorname{Lift}(f)$ from Section~\ref{sec:setting}.

\begin{proposition}\label{prop:L}
\begin{enumerate}[label=\textup{(\alph*)}]
\item $L=\{\eta\in\theta_S:\eta(\cO_X)\subseteq\cO_X\}$, and restriction to $\cO_X$ is an isomorphism of graded $\cO_X$-modules $L\xrightarrow{\ \sim\ }\Der_\C(\cO_X)$. Also $\Der_\C(\cO_X)\cong\Dlog X/G\,\Theta_T$ where $\operatorname{Lift}(f)=\Dlog X=\{\xi\in\Theta_T:\xi(G)\in(G)\}$.
\item \textup{(U1)} $\Hilb\NAe=\Hilb\theta(f)-\Hilb\theta_S-(1-t^{\mu k})\Hilb\Theta_T+\Hilb L$.
\end{enumerate}
\end{proposition}

\begin{proof}
(a) If $tf(\eta)=\omega f(\xi)$ then for $H\in\cO_T$, $\eta(H\circ f)=\sum_j(\partial_jH\circ f)\eta(f_j)=\sum_j(\partial_jH\cdot\xi_j)\circ f=\xi(H)\circ f\in\cO_X$. Conversely if $\eta$ preserves $\cO_X$, then $\eta|_{\cO_X}$ is a derivation of $\cO_X=\cO_T/(G)$, which lifts to some $\xi\in\Theta_T$ (choose $\xi_j$ with $\xi_j\circ f=\eta(f_j)$), and then $tf(\eta)=\omega f(\xi)$. Restriction $L\to\Der(\cO_X)$ is injective since a derivation of $\cO_S$ vanishing on $\cO_X$ vanishes on the common fraction field, and surjective by (Se) (\S\ref{sec:equation}: $\cO_X$ is a domain with normalisation $\cO_S$). Every derivation of $\cO_T/(G)$ lifts to $\Theta_T$, and the lifts form $\Dlog X$, with kernel $G\Theta_T$. Conversely, each field in $\Dlog X$ induces a derivation of $\cO_X$ and hence, by (Se), a field in $L$. Therefore $\operatorname{Lift}(f)=\Dlog X$.

(b) $tf$ is injective: $tf(\eta)=0$ gives $\eta(x_i)=\eta_{x_i}=0$ and then $\eta_yp_y=\eta_yq_y=0$, so $\eta_y=0$ as $p_y\ne0$. $\ker\omega f=\{\xi:\xi_j\circ f=0\ \forall j\}=(\ker f^*)\Theta_T=G\Theta_T$ (\S\ref{sec:equation}), so $\Hilb\omega f(\Theta_T)=(1-t^{\mu k})\Hilb\Theta_T$. Finally $tf(\theta_S)\cap\omega f(\Theta_T)=tf(L)$ by definition of $L$, and $\Hilb(A+B)=\Hilb A+\Hilb B-\Hilb(A\cap B)$.
\end{proof}

\subsection{The depth lemma and the seminormality criterion}\label{sec:seminormal}

\begin{lemma}\label{lem:seminormal}
$H^0_\fm(\cO_S/\cO_X)=0$. Consequently $h\in\cO_S$ lies in $\cO_X$ if and only if, for every $P\in X\setminus\{0\}$, its germ as a section of $f_*\cO_S^{an}$ at $P$ belongs to $\cO_{X,P}^{an}$, if and only if $h(s)=h(s')$ for every double point pair $(s,s')$ with $f(s)\ne0$, if and only if $h[12]\in I_{D^2(f)}$.
\end{lemma}

\begin{proof}
$\cO_X$ (hypersurface) and $\cO_S$ (regular, finite over $\cO_X$; \S\ref{sec:depth}) are CM $\cO_X$-modules of dimension $m$; \ref{D1} applied to $0\to\cO_X\to\cO_S\to\cO_S/\cO_X\to0$ gives $\depth(\cO_S/\cO_X)\ge m-1\ge1$, so $H^0_\fm(\cO_S/\cO_X)=0$. Put $U=X\setminus\{0\}$. If the section $h$ of $f_*\cO_S^{an}$ belongs to $\cO_X^{an}$ locally at every point of $U$, then Lemma~\ref{lem:comparison} applied to the (possibly non-graded) modules $M=\cO_X\subseteq N=\cO_X+\cO_Xh$ shows that the class of $h$ in $\cO_S/\cO_X$ is supported at $\fm$, hence lies in $H^0_\fm(\cO_S/\cO_X)=0$. By \S\ref{sec:local} (Prop.~\ref{prop:nc}(a), \ref{prop:crosscap}(a)) the local condition at $P\in X\setminus0$ is agreement of the branch functions on all nearby double-point pairs, using the whole multigerm over $P$. At cross-caps it is equivalently $h[12]\in I_{D^2}$ locally. Since $D^2(f)$ is reduced and the off-diagonal points $(s,s')$ with $f(s)\ne0$ are dense in $D^2(f)$ (Lemma~\ref{lem:Lambda}), the condition ``$h(s)=h(s')$ for all double point pairs off the origin'' is equivalent to $h[12]\in I_{D^2}$ (a function on the reduced $D^2$ vanishing on a dense set vanishes), and the cross-cap condition is included.
\end{proof}

\subsection{The conductor}\label{sec:conductor}

\begin{lemma}\label{lem:conductor}
The conductor $\fc=\{h\in\cO_S:h\cO_S\subseteq\cO_X\}$ of $\cO_X\subset\cO_S$ is $\fc=\Lambda\,\cO_S\cong\cO_S(-\deg\Lambda)$. Consequently $\cO_{D_X}:=\cO_X/\fc$ is a reduced Cohen--Macaulay ring of dimension $m-1$ with $\Hilb\cO_{D_X}=\Hilb\cO_X-t^{\deg\Lambda}\Hilb\cO_S$, and $D_X=\Spec\cO_{D_X}=f(D)$ is the double-point locus of $X$ with its reduced structure.
\end{lemma}

\begin{proof}
Let $h\in\cO_S$ be such that at every point $s\in S\setminus0$ the germ of $h$ lies in the local conductor (of the multigerm over $f(s)$). Then for every $g\in\cO_S$, $hg$ lies in $\cO_X$ locally on $U$, hence $hg\in\cO_X$ by Lemma~\ref{lem:seminormal}. So $h\in\fc$. Conversely elements of $\fc$ are in the local conductors. By \S\ref{sec:localsummary} the local conductor at every point of $S\setminus0$ is the reduced ideal of $D$ there. Hence $\fc=\{h:h|_{D\setminus0}=0\}=I(D)$, using that $D$ is reduced and every component of $D$ meets $S\setminus0$ densely. By Lemma~\ref{lem:Lambda-reduced}, $I(D)=(\Lambda)$. The rest: $\fc\cong\cO_S(-\deg\Lambda)$ is a maximal CM $\cO_X$-module (\S\ref{sec:depth}), so \ref{D1} applied to $0\to\fc\to\cO_X\to\cO_{D_X}\to0$ gives $\depth\cO_{D_X}\ge m-1=\dim\cO_{D_X}$; $\cO_{D_X}$ is reduced off the origin (\S\ref{sec:localsummary}), generically reduced and CM, hence reduced \ref{D4}. Its zero set is $f(D)$.
\end{proof}

\begin{corollary}[$m=2$]\label{cor:ainv}
$\cO_{D_X}$ is a $1$-dimensional graded CM ring with $a(\cO_{D_X})=\deg\Lambda-a-2v$.
\end{corollary}

\begin{proof}
The numerator of $\Hilb\cO_X-t^{\deg\Lambda}\Hilb\cO_S$ over $(1-t^a)(1-t^v)(1-t^{\mu v})(1-t^k)$ is $1-t^v-t^{\mu k}-t^{\deg\Lambda}+t^{\mu k-k+v}+t^{\mu k-\mu v+v}$ (using $\deg\Lambda+k+\mu v=\mu k+v$), of degree $\mu k$ since $\mu k>\max(\deg\Lambda,\mu k-k+v,\mu k-\mu v+v)$ for $k>v$, $\mu\ge2$. So $\deg\Hilb=\mu k-(a+v+\mu v+k)=\deg\Lambda-a-2v$, and \ref{D5} applies.
\end{proof}

\begin{remark}\label{rem:fitting}
The Fitting ideal $F_1=I_{\mu-1}(Z\cdot\Id-Q)$ satisfies $F_1\cO_S=\fc$ (Mond--Pellikaan \cite{MP89}. See also \cite{NBPSV} for the target double and triple point spaces of finite multigerms). This identifies $D_X$ with the double-point scheme of \cite{MP89} but is not used in the proofs.
\end{remark}

\section{A warm-up: the two-dimensional case}\label{sec:thmA}

This section proves Theorem~A. The proof of Theorem~B does not depend on it. Let $m=2$, so $D$ and $D_X$ are curves, $\Sigma=V(p_y,q_y)$ is $0$-dimensional (hence $=\{0\}$ set-theoretically), and every point of $X\setminus0$ is of type $A_0$ or $A_0^2$ (triple points and cross-caps are isolated, hence at $0$; \S\ref{sec:stable}).

\begin{theorem}[$=$ Theorem A, first part]\label{thm:A1}
$L=\Der(\cO_X)=\cO_X\eta_E+\Lambda\theta_S$ and $\cO_X\eta_E\cap\Lambda\theta_S=\Lambda\cO_S\eta_E$.
\end{theorem}

\begin{proof}
($\supseteq$) For $h\in\cO_X$, $h\eta_E(f_j)=h\deg(f_j)f_j\in\cO_X$. For $g\in\cO_S$ and $\eta\in\theta_S$, $\Lambda g\,\eta(f_j)\in\Lambda\cO_S=\fc\subseteq\cO_X$ (Lemma~\ref{lem:conductor}). (Intersection) $h\eta_E\in\Lambda\theta_S$ iff $\Lambda\mid h\,ax$ and $\Lambda\mid h\,vy$; a prime factor $\pi$ of the square-free $\Lambda$ cannot divide both $x$ and $y$, so $\pi\mid h$ for every $\pi$, i.e.\ $\Lambda\mid h$; and $\cO_X\cap\Lambda\cO_S=\Lambda\cO_S$ since $\Lambda\cO_S\subseteq\cO_X$.

($\subseteq$) Let $\eta\in L_\delta$ be homogeneous and $\xi\in(\Theta_T)_\delta$ with $\omega f(\xi)=tf(\eta)$ (Prop.~\ref{prop:L}; $\xi$ can be chosen homogeneous of degree $\delta$ since all maps are graded).

\emph{Step 1 (the Euler coefficient).} Let $P\in D_X\setminus0$. Since $f^{-1}(0)=\{0\}$ (\S\ref{sec:setting}), the preimages of $P$ lie in $S\setminus0$; $P$ is an $A_0^2$ point, $f^{-1}(P)=\{s,s'\}$, $f$ immersive at $s,s'$, with branches $X',X''$ transverse along $D_X$. From $\omega f(\xi)=tf(\eta)$: $\xi(P)=df_s\eta(s)\in T_PX'$ and $\xi(P)=df_{s'}\eta(s')\in T_PX''$, so $\xi(P)\in T_PX'\cap T_PX''=T_PD_X$ (transversality). By Lemma~\ref{lem:orbit} (applied to the reduced $\C^*$-invariant curve $D_X$, smooth at $P$), $T_PD_X=\C E(P)$, $E(P)\ne0$. Hence on the smooth curve $D_X\setminus0$ the field $\xi|_{D_X\setminus0}$ is a section of the tangent line bundle, of which $E$ is a nowhere-vanishing section. Thus
\[
\xi|_{D_X\setminus0}=\tilde g\,E,\qquad\tilde g\in\Gamma(D_X\setminus0,\cO_{D_X})_\delta ,
\]
homogeneous of degree $\delta$ because $\xi$ and $E$ are ($\deg E=0$). Since $df_s$ is injective and $df_s(\eta_E(s))=E(P)$, we get $\eta(s)=\tilde g(f(s))\,\eta_E(s)$ for every $s\in D\setminus0$.

\emph{Step 2 (degrees $\delta>\deg\Lambda-a-2v$).} By Corollary~\ref{cor:ainv} and \ref{D5}, $(\cO_{D_X})_\delta=\Gamma(D_X\setminus0,\cO_{D_X})_\delta$ for $\delta>\deg\Lambda-a-2v$, so $\tilde g\in(\cO_{D_X})_\delta$; lift it to $h\in(\cO_X)_\delta$ (the map $\cO_X\to\cO_{D_X}$ is a graded surjection). Then $\eta':=\eta-h\eta_E\in L_\delta$ satisfies $\eta'(s)=0$ for all $s\in D\setminus0$: its two components vanish on $D\setminus0$, hence on the reduced curve $D=V(\Lambda)$, hence are divisible by $\Lambda$. So $\eta'\in\Lambda\theta_S$ and $\eta\in\cO_X\eta_E+\Lambda\theta_S$.

\emph{Step 3 (degrees $\delta<\deg\Lambda-a-v$).} Put $\psi:=vy\,\eta_x-ax\,\eta_y\in(\cO_S)_{\delta+a+v}$. By Step~1, at every $s\in D\setminus0$, $(\eta_x(s),\eta_y(s))$ is proportional to $(ax,vy)(s)$, so $\psi(s)=0$. Hence $\psi$ vanishes on $D$ and $\Lambda\mid\psi$. As $\deg\psi<\deg\Lambda$, $\psi=0$: $vy\eta_x=ax\eta_y$. Since $x,y$ are coprime, $x\mid\eta_x$ and $y\mid\eta_y$; writing $\eta_x=axg'$ we get $\eta_y=vyg'$, so $\eta=g'\eta_E$ with $g'\in(\cO_S)_\delta$. Now $\eta\in L$ means $g'\deg(f_j)f_j\in\cO_X$ for $j=1,2,3$, i.e.\ $g'\fm_X\subseteq\cO_X$ where $\fm_X=(x,p,q)\cO_X$ is the irrelevant ideal of $\cO_X$. Thus the class of $g'$ in $\cO_S/\cO_X$ is killed by $\fm_X$, i.e.\ lies in $H^0_\fm(\cO_S/\cO_X)=0$ (Lemma~\ref{lem:seminormal}). So $g'\in\cO_X$ and $\eta\in\cO_X\eta_E$.

\emph{Step 4.} Since $v\ge1$, $\deg\Lambda-a-2v+1\le\deg\Lambda-a-v$, so every integer $\delta$ satisfies $\delta<\deg\Lambda-a-v$ or $\delta>\deg\Lambda-a-2v$.
\end{proof}

\begin{corollary}[$=$ Theorem A, second part]\label{cor:A2}
$\Hilb L=\Hilb\cO_X+t^{\deg\Lambda}(\Hilb\theta_S-\Hilb\cO_S)$, and $\Hilb\NAe=\mathcal H_2$.
\end{corollary}

\begin{proof}
$\cO_X\eta_E\cong\cO_X$ (torsion-free), $\Lambda\theta_S\cong\theta_S(-\deg\Lambda)$, and the intersection is $\Lambda\cO_S\eta_E\cong\cO_S(-\deg\Lambda)$. So $\Hilb L=\Hilb\cO_X+t^{\deg\Lambda}(\Hilb\theta_S-\Hilb\cO_S)$. Insert into (U1) with $\Hilb\cO_S=1/((1-t^a)(1-t^v))$, $\Hilb\theta_S=(t^{-a}+t^{-v})\Hilb\cO_S$, $\Hilb\theta(f)=(t^{-a}+t^{-\mu v}+t^{-k})\Hilb\cO_S$, $\Hilb\cO_T=1/((1-t^a)(1-t^{\mu v})(1-t^k))$, $\Hilb\Theta_T=(t^{-a}+t^{-\mu v}+t^{-k})\Hilb\cO_T$, $\Hilb\cO_X=(1-t^{\mu k})\Hilb\cO_T$. This is $\mathcal H_2$ as defined in Section~\ref{sec:setting}. The finite-sum expansion is given in Section~\ref{sec:expanded} and verified in Appendix~\ref{sec:algebra}.
\end{proof}

\section{\texorpdfstring{The case $m=3$: rigidity on the double-point surface}{The case m=3: rigidity on the double-point surface}}\label{sec:rigidity}

Now $m=3$. By \S\ref{sec:stable}, $X\setminus0$ consists of points of type $A_0$, $A_0^2$ (a dense open subset of the surface $D_X\setminus0$), $A_0^3$ (the triple curve $T_X\setminus0$, possibly empty) and $S_0$ (the cross-cap curve $C_X\setminus0$, $C_X=f(\Sigma)$). Let $I_{C_X}\subset\cO_{D_X}$ be the reduced ideal of $C_X$ and
\[
\Der(\cO_{D_X};C_X):=\{\delta\in\Der_\C(\cO_{D_X}):\ \delta(I_{C_X})\subseteq I_{C_X}\}.
\]

\begin{lemma}\label{lem:Phi}
Every $\xi\in\Dlog X$ induces a derivation of $\cO_{D_X}=\cO_X/\fc$, and the induced derivation lies in $\Der(\cO_{D_X};C_X)$. This defines a graded $\cO_X$-linear map $\Phi:L\cong\Der(\cO_X)\to\Der(\cO_{D_X};C_X)$ with $\ker\Phi=\Lambda\theta_S$.
\end{lemma}

\begin{proof}
Let $\eta\in L$ correspond to $\xi$. For $h\in\fc$ and $g\in\cO_S$: $\eta(hg)=\eta(h)g+h\eta(g)$, where $\eta(hg)\in\cO_X$ (as $hg\in\fc\subseteq\cO_X$) and $h\eta(g)\in\fc\cO_S\subseteq\cO_X$. Hence $\eta(h)g\in\cO_X$ for all $g$, i.e.\ $\eta(h)\in\fc$. So $\eta$ (equivalently $\xi$) preserves $\fc$ and induces $\bar\xi\in\Der(\cO_{D_X})$. Tangency to $C_X$: for $g\in I_{C_X}$, the function $\bar\xi(g)$ vanishes on $C_X\setminus0$ by Prop.~\ref{prop:crosscap}(d) (at an $S_0$ point the induced derivations are $\cO$-combinations of $Y\partial_Y,\partial_w$, which preserve the local reduced ideal $(Y)$ of $C_X$), hence on $C_X$ (closure), hence $\bar\xi(g)\in I_{C_X}$ (reduced ideal). Kernel: $\Phi(\eta)=0$ iff $\xi(\cO_X)\subseteq\fc$ iff $\eta(f_j)\in\Lambda\cO_S$ for all $j$. If $\eta\in\Lambda\theta_S$ this holds. Conversely, $\Lambda\mid\eta_{x_1},\eta_{x_2}$ and then $\Lambda\mid\eta(p)-\eta_{x}\cdot p_x=\eta_yp_y$ and $\Lambda\mid\eta_yq_y$; a prime factor $\pi$ of $\Lambda$ not dividing $\eta_y$ would divide $p_y$ and $q_y$, so $V(\pi)\subseteq\Sigma$, impossible as $\dim V(\pi)=2>1=\dim\Sigma$. Hence $\Lambda\mid\eta_y$ ($\Lambda$ square-free) and $\eta\in\Lambda\theta_S$.
\end{proof}

\begin{theorem}[rigidity]\label{thm:rigidity}
$\Phi$ induces an isomorphism $L/\Lambda\theta_S\xrightarrow{\ \sim\ }\Der(\cO_{D_X};C_X)$.
\end{theorem}

\begin{proof}
Injectivity is Lemma~\ref{lem:Phi}. For surjectivity we use Lemma~\ref{lem:depth2}(c) with $M=\Phi(L)\subseteq N=\Der(\cO_{D_X};C_X)$, both submodules of the free $\cO_{D_X}$-module $\Theta_T\otimes\cO_{D_X}\cong\cO_{D_X}^4$ (a derivation of $\cO_{D_X}$ is determined by its values on the four coordinates).

\emph{Local surjectivity on $U$.} At $P\in D_X\setminus0$ of type $A_0^2$, $A_0^3$, $S_0$, Propositions~\ref{prop:nc}(d),(e) and \ref{prop:crosscap}(d) show that the image of $\Dlog X_P\to\Der(\cO_{D_X})_P$ is all of $\Der(\cO_{D_X})_P$ (types $A_0^2,A_0^3$, where $I_{C_X}$ is the unit ideal locally) resp.\ $\Der(\cO_{D_X};C_X)_P$ (type $S_0$). By Lemma~\ref{lem:comparison}, $\widetilde M|_U=\widetilde N|_U$.

\emph{Extension across the vertex.} The module
$L\cong\Hom_{\cO_X}(\Omega_{\cO_X},\cO_X)$ has depth at least two by \ref{D3}, since $\cO_X$ is Cohen--Macaulay of dimension three. The module $\Lambda\theta_S$ is free over $\cO_S$ up to grading shifts, and hence has depth three over $\cO_X$. The depth lemma gives
\[
\depth M=\depth(L/\Lambda\theta_S)
\ge\min\{\depth L,\depth(\Lambda\theta_S)-1\}\ge2.
\]
Meanwhile $N\subseteq\cO_{D_X}^4$ has depth at least one by \ref{D2}. Lemma~\ref{lem:depth2}(c) therefore gives $M=N$.
\end{proof}

\begin{remark}\label{rem:m2fails}
The depth argument is exactly what fails for $m=2$. There $\dim D_X=1$, $\theta_S$ has depth $2$, and $L/\Lambda\theta_S$ is only guaranteed depth $\ge1$. Indeed, for $m=2$ the image of $\Phi$ is the proper submodule $\cO_{D_X}E\subsetneq\Der(\cO_{D_X})$, and Step~3 of Theorem~\ref{thm:A1} is needed instead. Likewise for $m=3$ the inclusion $\Der(\cO_{D_X};C_X)\subseteq\Der(\cO_{D_X})$ is strict in general, the quotient being supported on $C_X$ (Prop.~\ref{prop:crosscap}(d)).
\end{remark}

\section{\texorpdfstring{The normalisation of the double-point surface ($m=3$)}{The normalisation of the double-point surface (m=3)}}\label{sec:normalisation}

\subsection{\texorpdfstring{$D^2/S_2$ and its invariant vector fields}{The quotient double-point surface and its invariant vector fields}}\label{sec:invariants}

Let $R=\C[x_1,x_2,y_1,y_2]$ with the swap $\sigma:y_1\leftrightarrow y_2$, $R^{S_2}=\C[x_1,x_2,s_1,s_2]$, $s_1=y_1+y_2$ (weight $v$), $s_2=y_1y_2$ (weight $2v$), $\Delta:=s_1^2-4s_2=(y_1-y_2)^2$. $R=R^{S_2}\oplus R^{S_2}(y_1-y_2)$ (invariant and anti-invariant parts), and $I:=I_{D^2}=(p[12],q[12])$ is generated by invariants; write $P_{12},Q_{12}\in R^{S_2}$ for $p[12],q[12]$ expressed in $x,s$ and $I':=(P_{12},Q_{12})\subset R^{S_2}$. We write $\bD$ for the class of $\Delta$ in $\cO_{\tD}$.

\begin{lemma}\label{lem:normalisation}
\begin{enumerate}[label=\textup{(\alph*)}]
\item $I\cap R^{S_2}=I'$ and $\cO_{D^2}^{S_2}=R^{S_2}/I'=:\cO_{\tD}$. It is a complete intersection of dimension $2$ (hence CM and Gorenstein), a normal domain (Lemma~\ref{lem:D2normal}), and $\tD:=V(P_{12},Q_{12})\subset\C^4_{x,s}$ is smooth off the origin, with
\[
\Hilb\cO_{\tD}=\frac{(1-t^{(\mu-1)v})(1-t^{k-v})}{(1-t^{a_1})(1-t^{a_2})(1-t^v)(1-t^{2v})}.
\]
\item The ring homomorphism $\cO_{D_X}\to\cO_{\tD}$ given by $X_i\mapsto x_i$, $Y\mapsto\bar P:=\tfrac12(p(x,y_1)+p(x,y_2))$, $Z\mapsto\bar Q:=\tfrac12(q(x,y_1)+q(x,y_2))$ is the composite of the two injections
\[
\cO_{D_X}=\cO_X/\fc\ \hookrightarrow\ \cO_S/\Lambda\cO_S=\cO_D\ \hookrightarrow\ \cO_{D^2},
\]
lands in the invariants $\cO_{D^2}^{S_2}=\cO_{\tD}$, and $\cO_{D_X}\to\cO_{\tD}$ is the normalisation of $\cO_{D_X}$. In particular $\cO_{D_X}$ is a domain.
\item The branch locus of $D^2\to\tD$ is the curve $\tC:=V(\bD)\subset\tD$, image of the fixed curve $F=D^2\cap\{y_1=y_2\}\cong\Sigma$; $(\Delta,P_{12},Q_{12})$ is a reduced ideal of $R^{S_2}$, so $\bD$ is a reduced equation of $\tC$ and $I_{\tC}=(\bD)$. Set-theoretically $\tC$ is the preimage of $C_X$ under $\tD\to D_X$, i.e.\ $\sqrt{I_{C_X}\cO_{\tD}}=(\bD)$, and $I_{C_X}=\cO_{D_X}\cap(\bD)$. The map $\tD\to D_X$ is an isomorphism near $\tC\setminus0$. \emph{The scheme-theoretic inverse image $I_{C_X}\cO_{\tD}$ can be strictly smaller than $(\bD)$} (Example~\ref{ex:embedded}). Only the reduced statements are used below.
\item The $S_2$-invariant vector fields on $\C^4_{x,y}$ form the free $R^{S_2}$-module $\Dlog\Delta$ with basis $\partial_{x_1},\partial_{x_2},\ \partial_+:=\partial_{y_1}+\partial_{y_2}=2\partial_{s_1}+s_1\partial_{s_2},\ \partial_\Delta:=-(y_1-y_2)(\partial_{y_1}-\partial_{y_2})=\Delta\partial_{s_2}$.
\item $\Der(\cO_{D^2})^{S_2}=\{\xi\in\Dlog\Delta:\xi(I')\subseteq I'\}/I'\Dlog\Delta=\Der(\cO_{\tD};\tC)=:\tK$, the derivations of $\cO_{\tD}$ preserving $(\bD)$.
\end{enumerate}
\end{lemma}

\begin{proof}
(a) If $u=\alpha P_{12}+\beta Q_{12}\in I$ is invariant ($\alpha,\beta\in R$), write $\alpha=\alpha^++(y_1-y_2)\alpha^-$ etc.. The anti-invariant part $(y_1-y_2)(\alpha^-P_{12}+\beta^-Q_{12})$ must vanish, so $u=\alpha^+P_{12}+\beta^+Q_{12}\in I'$. $\dim\tD=\dim D^2=2$ (the quotient map is finite), so $P_{12},Q_{12}$ is a regular sequence in the $4$-dimensional polynomial ring $R^{S_2}$ \ref{D4}. The Hilbert series follows. Normal domain: Lemma~\ref{lem:D2normal}. Off the origin $D^2$ is smooth and $\sigma$ acts on the smooth surface with smooth fixed curve $F\setminus0$ (a reflection in suitable local coordinates), so the quotient is smooth off $0$.

(b) The first arrow is induced by $\cO_X\subseteq\cO_S$ and is injective because $\fc=\Lambda\cO_S=\cO_X\cap\Lambda\cO_S$ (Lemma~\ref{lem:conductor}: $\Lambda\cO_S\subseteq\cO_X$). The second is induced by $\pi_1^*:\cO_S\to\cO_{D^2}$, $(x,y_1)\mapsto(x,y_1)$, which kills $\Lambda$ (Lemma~\ref{lem:Lambda}(c)). It is injective because $D$ is reduced (Lemma~\ref{lem:Lambda-reduced}) and $\pi_1:D^2\to D$ is surjective, so a function on $D$ vanishing on $D^2$ vanishes on $D$. The composite sends $X_i,Y,Z$ to the classes of $x_i,p(x,y_1),q(x,y_1)$, which on $D^2$ equal $x_i,\bar P,\bar Q$ (since $p(x,y_1)=p(x,y_2)$ and $q(x,y_1)=q(x,y_2)$ on $D^2$). These are $S_2$-invariant, so the image lies in $\cO_{\tD}$. The map $\cO_{D_X}\to\cO_{\tD}$ is finite (integral extension) and generically one-to-one: a generic point of $D_X$ is an $A_0^2$ point with exactly one unordered pair $\{s,s'\}$ over it, so $\tD\to D_X$ is birational; $\cO_{\tD}$ is a normal domain (a). Hence $\cO_{D_X}\subseteq\cO_{\tD}$ is the normalisation of the reduced ring $\cO_{D_X}$ (Lemma~\ref{lem:conductor}), and $\cO_{D_X}$, a subring of a domain, is a domain.

(c) On $V(\Delta)$, $s_2=s_1^2/4$ and $s_1=2y$. Then $P_{12}(x,2y,y^2)=p[12](x,y,y)=p_y(x,y)$, $Q_{12}\mapsto q_y$, so $R^{S_2}/(\Delta,P_{12},Q_{12})\cong\C[x,y]/(p_y,q_y)=\cO_\Sigma$, reduced (MM: $\Sigma$ is an ICIS curve). The image of $F$ in $\tD$ is $V(\bD)$ and maps onto $f(\Sigma)=C_X$. Conversely a point of $\tD\setminus0$ over $C_X$ is an unordered pair $\{s,s'\}$ of preimages of a cross-cap point $P\in C_X\setminus0$, and by (St) $f^{-1}(P)$ is a single non-immersive point, so $s=s'$ and the pair lies on $\tC$. Hence $\tC$ is the set-theoretic preimage of $C_X$, and $\sqrt{I_{C_X}\cO_{\tD}}=I_{\tC}=(\bD)$ (both sides are the reduced ideal of $\tC$ in the reduced ring $\cO_{\tD}$). $I_{C_X}=\cO_{D_X}\cap(\bD)$: both are the functions on $D_X$ vanishing on $C_X$, i.e.\ on the image of $\tC$. Near a point of $C_X\setminus0$ (type $S_0$), $\tD\to D_X$ is an isomorphism by Prop.~\ref{prop:crosscap}(e) and the intrinsic nature of both sides.

(d) A field $\xi=A_1\partial_{x_1}+A_2\partial_{x_2}+A_3\partial_{y_1}+A_4\partial_{y_2}$ is invariant iff $\sigma(A_1)=A_1$, $\sigma(A_2)=A_2$, $\sigma(A_3)=A_4$. Write $A_3=\alpha+(y_1-y_2)\beta$ with $\alpha,\beta\in R^{S_2}$. Then $A_4=\alpha-(y_1-y_2)\beta$ and $\xi=A_1\partial_{x_1}+A_2\partial_{x_2}+\alpha\partial_+-\beta\partial_\Delta$, uniquely. In $s$-coordinates $\partial_{y_i}=\partial_{s_1}+y_{3-i}\partial_{s_2}$, giving the stated expressions; $\partial_+(\Delta)=4s_1-4s_1=0$, $\partial_\Delta(\Delta)=-4\Delta$, so these fields are tangent to $\Delta$, and $\det(\partial_{x_1},\partial_{x_2},\partial_+,\partial_\Delta)=2\Delta$, so by (S) they form a basis of $\Dlog\Delta$ (directly: a field tangent to $V(\Delta)$ has $s_2$-component divisible by $\Delta$ after subtracting a multiple of $\partial_+$, because $\Delta$ is reduced and $\partial_{s_2}\Delta=-4$ is a unit).

(e) Taking $S_2$-invariants is exact in characteristic $0$, so $\Der(\cO_{D^2})^{S_2}=\Dlog I^{S_2}/(I\Theta)^{S_2}$, where $\Dlog I=\{\xi:\xi(I)\subseteq I\}$. $\Dlog I^{S_2}=\{\xi\text{ invariant}:\xi(P_{12}),\xi(Q_{12})\in I\}$, and for invariant $\xi$ the values $\xi(P_{12}),\xi(Q_{12})$ are invariant, so lie in $I\cap R^{S_2}=I'$. $(I\Theta)^{S_2}=P_{12}\Theta^{S_2}+Q_{12}\Theta^{S_2}=I'\Dlog\Delta$ by averaging ($\tfrac12(\eta+\sigma\eta)$ of $P_{12}\zeta+Q_{12}\zeta'$ is $P_{12}\tfrac12(\zeta+\sigma\zeta)+\dots$). This proves the first equality. For the second: $\Der(\cO_{\tD})=\{\xi\in\Theta_{x,s}:\xi(I')\subseteq I'\}/I'\Theta_{x,s}$, and $\Dlog\Delta\cap I'\Theta_{x,s}=I'\Dlog\Delta$ (write $\xi=A_1\partial_{x_1}+A_2\partial_{x_2}+B_1\partial_++B_2\partial_\Delta$; $\xi\in I'\Theta$ forces $A_1,A_2,2B_1\in I'$ and $B_1s_1+B_2\Delta\in I'$, hence $B_2\Delta\in I'$, hence $B_2\in I'$ as $\Delta\notin I'$ and $I'$ is prime), so $\tK\hookrightarrow\Der(\cO_{\tD})$ with image the derivations induced by fields tangent to $\Delta$. Finally a derivation $\bar\xi$ of $\cO_{\tD}$ with $\bar\xi(\bD)\in(\bD)$ lifts to $\xi\in\Theta_{x,s}$ with $\xi(\Delta)=u\Delta+i$, $i\in I'$; replacing $\xi$ by $\xi+\tfrac i4\partial_{s_2}$ (which changes $\xi$ by an element of $I'\Theta$ and changes $\xi(\Delta)$ by $\tfrac i4\cdot(-4)=-i$) gives a lift tangent to $\Delta$. (Here we used that $\bD$ is a nonzerodivisor of $\cO_{\tD}$, which is a domain, so ``$\bar\xi(\bD)\in(\bD)$'' is meaningful without ambiguity.)
\end{proof}

\begin{remark}[Euler fields]\label{rem:Eulers} The Euler derivation $E$ of $\cO_{\tD}$ is induced by the field $\sum a_ix_i\partial_{x_i}+vs_1\partial_{s_1}+2vs_2\partial_{s_2}$ on $\C^4_{x,s}$. It restricts to the Euler derivation of $\cO_{D_X}$ (both multiply a homogeneous element by its degree), which is the derivation induced by the target Euler field $E$ of \S\ref{sec:setting} via Lemma~\ref{lem:Phi}. By Lemma~\ref{lem:seidenberg-graded} these are the unique extensions of one another, so the single letter $E$ is unambiguous.
\end{remark}

\begin{lemma}\label{lem:Der-DX-in-K}
$\Der(\cO_{D_X};C_X)=\{\delta\in\tK:\ \delta(\cO_{D_X})\subseteq\cO_{D_X}\}$, where $\Der(\cO_{D_X})\hookrightarrow\Der(\cO_{\tD})$ is the Seidenberg extension \textup{(Se)}.
\end{lemma}

\begin{proof}
$\Der(\cO_{D_X})=\{\delta\in\Der(\cO_{\tD}):\delta(\cO_{D_X})\subseteq\cO_{D_X}\}$ by (Se) (applied to the domain $\cO_{D_X}$ with normalisation $\cO_{\tD}$, Lemma~\ref{lem:normalisation}(b)) and injectivity of restriction. It remains to see that for such $\delta$, $\delta(I_{C_X})\subseteq I_{C_X}$ iff $\delta(\bD)\in(\bD)$. We use only the reduced ideals $I_{C_X}=\cO_{D_X}\cap(\bD)$ and $I_{\tC}=(\bD)$ of Lemma~\ref{lem:normalisation}(c). If $\delta$ preserves $(\bD)$ then it preserves $\cO_{D_X}\cap(\bD)=I_{C_X}$. Conversely suppose $\delta$ preserves $I_{C_X}$, and let $g\in(\bD)$. Near each point of $\tC\setminus0$ the rings $\cO_{D_X}$ and $\cO_{\tD}$ coincide analytically and $I_{C_X}$, $(\bD)$ generate the same (reduced) local ideal of the curve. So $\delta(g)$ vanishes on $\tC\setminus0$, hence on its closure $\tC$ (every component of $\tC$ meets $\tC\setminus0$ densely, being a positive-dimensional cone), hence lies in the reduced ideal $(\bD)$. No equality of scheme-theoretic inverse images is needed.
\end{proof}

\subsection{\texorpdfstring{The conductor of the normalisation of $D_X$}{The conductor of the normalisation of the double-point surface}}\label{sec:conductorD}

\begin{lemma}\label{lem:conductorD}
$H^0_\fm(\cO_{\tD}/\cO_{D_X})=0$, and the conductor of $\cO_{D_X}\subset\cO_{\tD}$ is $\fc_D=\bL\,\cO_{\tD}\cong\cO_{\tD}(-\deg\Lambda_3)$.  If the triple curve is empty, $\fc_D=\cO_{\tD}$, i.e.\ $\cO_{D_X}=\cO_{\tD}$ is normal.
\end{lemma}

\begin{proof}
Both rings are CM $\cO_{D_X}$-modules of dimension $2$ (Lemmas~\ref{lem:conductor} and \ref{lem:normalisation}(a) with \S\ref{sec:depth}), so \ref{D1} gives $\depth(\cO_{\tD}/\cO_{D_X})\ge1$, i.e.\ $H^0_\fm(\cO_{\tD}/\cO_{D_X})=0$.  Hence, as in Lemma~\ref{lem:conductor}: if $h\in\cO_{\tD}$ lies in the local conductor at every point of $\tD\setminus0$, then for every $g\in\cO_{\tD}$ the class of $hg$ in $\cO_{\tD}/\cO_{D_X}$ is supported at the origin (Lemma~\ref{lem:comparison}), hence zero, so $hg\in\cO_{D_X}$ and $h\in\fc_D$. Conversely elements of $\fc_D$ lie in the local conductors. By \S\ref{sec:localsummary} the local conductor is the unit ideal except over $T_X\setminus0$, where (Prop.~\ref{prop:nc}(e)) it is the reduced ideal of the triple line in each plane, i.e.\ the reduced ideal of $\tT_X$. Hence $\fc_D=I(\tT_X)$, and by Lemma~\ref{lem:Lambda3} this is $(\bL)$ (the unit ideal in the empty case).
\end{proof}

\subsection{Freeness}\label{sec:freeness}

Let $\vol=dx_1\wedge dx_2\wedge ds_1\wedge ds_2$ and define, on $\C^4_{x,s}$, the vector field $\zeta$ and the bivector $\varepsilon$ by
\[
\iota_\zeta\vol=\tfrac12\,d\Delta\wedge dP_{12}\wedge dQ_{12},\qquad \iota_\varepsilon\vol=dP_{12}\wedge dQ_{12}
\]
(i.e.\ $\zeta=\tfrac12\star(d\Delta\wedge dP_{12}\wedge dQ_{12})$, the $4$-dimensional cross product of the three gradients: $\zeta_l=\tfrac12(-1)^{l+1}\det$ of the $3\times3$ minor of $(\nabla\Delta;\nabla P_{12};\nabla Q_{12})$ omitting column $l$; and $\varepsilon=\star(dP_{12}\wedge dQ_{12})$).

\begin{proposition}[freeness]\label{prop:freeness}
\begin{enumerate}[label=\textup{(\alph*)}]
\item $\zeta(\Delta)=\zeta(P_{12})=\zeta(Q_{12})=0$; $\zeta$ is weighted homogeneous of degree $e_\zeta=\deg\Delta+\deg P_{12}+\deg Q_{12}-(a_1+a_2+3v)=k+(\mu-3)v-a_1-a_2$; $\zeta\in\Dlog\Delta$ and its class lies in $\tK$.
\item  On $\tD$ (i.e.\ modulo $I'$): $E\wedge\zeta=\pm v\,\bD\,\varepsilon$, the sign depending only on the orientation convention for $\iota$.
\item $\varepsilon$ restricts to a nowhere-vanishing section of $\wedge^2T(\tD\setminus0)$.
\item $\tK=\cO_{\tD}E\oplus\cO_{\tD}\zeta$ is free of rank $2$; $\Hilb\tK=(1+t^{e_\zeta})\Hilb\cO_{\tD}$.
\end{enumerate}
\end{proposition}

\begin{proof}
(a) $\zeta(g)\,\vol=dg\wedge\iota_\zeta\vol=\tfrac12dg\wedge d\Delta\wedge dP_{12}\wedge dQ_{12}=0$ for $g\in\{\Delta,P_{12},Q_{12}\}$ (a repeated factor). The degree count is the weight of the $3$-form minus that of $\vol$. Since $\zeta(\Delta)=0$, $\zeta\in\Dlog\Delta$. Since $\zeta(I')=0$, its class is in $\tK$.

(b) For a $3$-form $\omega=\alpha\wedge\beta\wedge\gamma$, $\iota_E\omega=\alpha(E)\,\beta\wedge\gamma-\beta(E)\,\alpha\wedge\gamma+\gamma(E)\,\alpha\wedge\beta$. With $(\alpha,\beta,\gamma)=(d\Delta,dP_{12},dQ_{12})$ and $E(\Delta)=2v\Delta$, $E(P_{12})=(\mu-1)vP_{12}$, $E(Q_{12})=(k-v)Q_{12}$ (Euler relations. Here $E$ is the field of Remark~\ref{rem:Eulers} on $\C^4_{x,s}$):
\[
\iota_E\iota_\zeta\vol=\tfrac12\iota_E(d\Delta\wedge dP_{12}\wedge dQ_{12})\equiv v\Delta\,dP_{12}\wedge dQ_{12}=v\Delta\,\iota_\varepsilon\vol\pmod{I'\Omega^2}.
\]
Since $\iota_E\iota_\zeta\vol=\pm\iota_{E\wedge\zeta}\vol$ (the sign is the convention) and $B\mapsto\iota_B\vol$ is an $\cO$-linear isomorphism $\wedge^2\Theta\to\Omega^2$, we get $E\wedge\zeta\equiv\pm v\Delta\varepsilon$ modulo $I'$. Only the fact that $E\wedge\zeta$ is a unit multiple of $\bD\varepsilon$ is used.

(c) At a point $z\in\tD\setminus0$ the differentials $dP_{12},dQ_{12}$ are independent (smoothness of the complete intersection), so $\varepsilon(z)\ne0$; and for any $1$-form $\alpha\in\operatorname{span}(dP_{12},dQ_{12})$, $(\iota_\alpha\varepsilon)\,\vol=\pm\alpha\wedge dP_{12}\wedge dQ_{12}=0$, so $\varepsilon(z)$ is annihilated by $dP_{12},dQ_{12}$, i.e.\ lies in $\wedge^2(\ker dP_{12}\cap\ker dQ_{12})=\wedge^2T_z\tD$, a line. So $\varepsilon$ is a nowhere-vanishing section of $\wedge^2T(\tD\setminus0)$.

(d) $E,\zeta\in\tK\subseteq\Dlog{\tC}$ on the smooth surface $\tD\setminus0$. \emph{Local basis on the punctured surface.} Let $z\in\tD\setminus0$. If $z\notin\tC$, then by (b),(c) $\det(E,\zeta)$ in any local chart is a unit, so $E,\zeta$ form a local basis of the tangent sheaf $=\Dlog{\tC}$ there. If $z\in\tC\setminus0$, then $\tC$ is smooth at $z$ (it is the reduced curve $\cong\Sigma\setminus0$, an ICIS curve, smooth off the origin), so there are local coordinates $(u,w)$ on $\tD$ at $z$ with $\bD=(\text{unit})\cdot u$, and $\Dlog{\tC}_z=\cO_z\,u\partial_u\oplus\cO_z\,\partial_w$ (a field $A\partial_u+B\partial_w$ preserves $(u)$ iff $u\mid A$). Write $E$ and $\zeta$ in the basis $(u\partial_u,\partial_w)$: the coefficient matrix $C$ satisfies $\det(E,\zeta)_{(\partial_u,\partial_w)}=u\cdot\det C$. By (b),(c) the left side is a unit times $\bD$, i.e.\ a unit times $u$, so $\det C$ is a unit and $E,\zeta$ form a basis of $\Dlog{\tC}_z$. (This is the two-line proof of (S) in the case of a smooth divisor on a smooth surface.) Hence $\cO E\oplus\cO\zeta\to\tK$ is injective (two fields independent on a dense open set are independent over the domain $\cO_{\tD}$) and an isomorphism of sheaves on $U=\tD\setminus0$.

\emph{Extension across the vertex.} The proposed module $M=\cO_{\tD}E\oplus\cO_{\tD}\zeta$ is free over the two-dimensional Cohen--Macaulay ring $\cO_{\tD}$, so has depth two. The module $\tK$ embeds in $\cO_{\tD}^4$, hence has depth at least one. Lemma~\ref{lem:depth2}(c) gives $M=\tK$. The Hilbert series follows from $\deg E=0$ and $\deg\zeta=e_\zeta$.
\end{proof}

\begin{remark}\label{rem:zeta-y}
In $y$-coordinates, $\zeta$ is the invariant combination $H_{\text{omit }y_1}+H_{\text{omit }y_2}$ of the Hamiltonian (cross-product) fields of the ICIS $D^2$ ($\sigma(H_{\text{omit }y_1})=H_{\text{omit }y_2}$, $\sigma(H_{\text{omit }x_j})=-H_{\text{omit }x_j}$), because $\vol_s=(y_1-y_2)\vol_y$ and $\tfrac12d\Delta=(y_1-y_2)(dy_1-dy_2)$. The degree count in (a) alone forces the shape of (b): $E\wedge\zeta=\varphi\varepsilon$ with $\deg\varphi=e_\zeta-\deg\varepsilon=2v=\deg\Delta$, and $\varphi$ vanishes on $\tC$ (both $E$ and $\zeta$ are tangent to the curve there), so $\varphi=c\bD$ with $c\in\C$.
\end{remark}

\subsection{Descent}\label{sec:descent}

\begin{proposition}[descent]\label{prop:descent}
$\Der(\cO_{D_X};C_X)=\cO_{D_X}E+\bL\tK=\cO_{D_X}E\oplus\bL\cO_{\tD}\zeta\cong\cO_{D_X}\oplus\cO_{\tD}(-d_0)$, with $d_0=\deg\Lambda_3+e_\zeta=(\mu-1)(k-v)-a_1-a_2$ as in Theorem~B.
\end{proposition}

\begin{proof}
We work inside $\tK$ via Lemma~\ref{lem:Der-DX-in-K}. ($\supseteq$) $E$ preserves $\cO_{D_X}$ and $I_{C_X}$. For $\xi\in\tK$ and $h\in\cO_{D_X}$: $\bL\xi(h)\in\bL\cO_{\tD}=\fc_D\subseteq\cO_{D_X}$ (Lemma~\ref{lem:conductorD}), and for $g\in I_{C_X}$, $\bL\xi(g)\in\fc_D\cap I_{\tC}\subseteq\cO_{D_X}\cap(\bD)=I_{C_X}$. Directness: by Prop.~\ref{prop:freeness}(d), $\cO_{D_X}E\cap\bL\tK=\{hE:h\in\cO_{D_X}\cap\bL\cO_{\tD}\}=\bL\cO_{\tD}E$ (as $\bL\cO_{\tD}\subseteq\cO_{D_X}$), so $\cO_{D_X}E+\bL\tK=\cO_{D_X}E+\bL\cO_{\tD}E+\bL\cO_{\tD}\zeta=\cO_{D_X}E\oplus\bL\cO_{\tD}\zeta$, and $\bL\cO_{\tD}\zeta\cong\cO_{\tD}(-\deg\Lambda_3-e_\zeta)$ ($\bL$ is a nonzerodivisor of the domain $\cO_{\tD}$, and $\zeta$ is part of a basis).

($\subseteq$, locally on $U$.)  At points of $D_X\setminus0$ not on $T_X$, $\cO_{D_X}=\cO_{\tD}$ locally and $\bL$ is a unit (\S\ref{sec:localsummary}), so both sides equal $\tK$ locally (the left side being $\Der(\cO_{D_X};C_X)=\tK$ there by Lemma~\ref{lem:Der-DX-in-K}). In the empty-triple-curve case this already covers all of $U$. At a point $P\in T_X\setminus0$ (type $A_0^3$), in the notation of Prop.~\ref{prop:nc}(e) ($n=3$, one transverse coordinate $w$): $\cO_{D_X}=\{(h_i)\in\prod_i\C\{u_i,w\}:h_i|_T\text{ equal}\}$, $\Der(\cO_{D_X})=\cO_{D_X}\partial_w\oplus\bigoplus_iu_i\C\{u_i,w\}\partial_{u_i}$ (and $C_X$ is absent, so $\Der(\cO_{D_X};C_X)=\Der(\cO_{D_X})$), $\tK=\Der(\bigsqcup\Pi_i)=\bigoplus_i(\C\{u_i,w\}\partial_{u_i}\oplus\C\{u_i,w\}\partial_w)$, and $\bL=(\text{unit})\cdot u_i$ on $\Pi_i$ (Lemma~\ref{lem:Lambda3}, Prop.~\ref{prop:nc}(e)), so $\bL\tK=\bigoplus_iu_i(\C\{u_i,w\}\partial_{u_i}\oplus\C\{u_i,w\}\partial_w)$. The Euler field is tangent to the curve $T_X$ and nonvanishing along it (Lemma~\ref{lem:orbit} applied to $T_X$; along $T=V(u_1,u_2,u_3)$ the fields $u_i\partial_{u_i}$ vanish, so the tangent direction of $T$ is $\partial_w$), so $E=\alpha\partial_w+\sum_i\beta_iu_i\partial_{u_i}$ with $\alpha\in\cO_{D_X}$ nonvanishing at $P$, hence a unit of the local ring. Then $\alpha\partial_w=E-\sum\beta_iu_i\partial_{u_i}\in\cO_{D_X}E+\bL\tK$, hence $\partial_w=\alpha^{-1}(\alpha\partial_w)$ is in the right-hand side, and so is $\bigoplus_iu_i\C\{u_i,w\}\partial_{u_i}$. Therefore the right-hand side contains $\cO_{D_X}\partial_w\oplus\bigoplus_iu_i\C\{u_i,w\}\partial_{u_i}=\Der(\cO_{D_X})$. With ($\supseteq$) the two sides coincide locally at $P$.

(Global.) Put $N=\Der(\cO_{D_X};C_X)$ and
\[
M=\cO_{D_X}E\oplus\bL\cO_{\tD}\zeta
\cong\cO_{D_X}\oplus\cO_{\tD}(-d_0).
\]
The inclusion $M\subseteq N$ was proved above, and local equality implies that $N/M$ has finite length by Lemma~\ref{lem:comparison}. Both summands of $M$ are Cohen--Macaulay of dimension two over $\cO_{D_X}$, so $\depth M=2$. The embedding $N\subseteq\tK\subseteq\cO_{\tD}^4$ gives $\depth N\ge1$. Lemma~\ref{lem:depth2}(c) now gives equality.
\end{proof}

\section{Hilbert series, Mond's equality, and an example}\label{sec:proofB}

\subsection{Proof of Theorem B}\label{sec:proofB1}

The module identity is Theorem~\ref{thm:rigidity} together with Proposition~\ref{prop:descent}. For the Hilbert series we have
\begin{center}
\begin{tabular}{lcl} 
$\Hilb L=t^{\deg\Lambda}\Hilb\theta_S+\Hilb\Der(\cO_{D_X};C_X)$ & & by Theorem~\ref{thm:rigidity} and Lemma~\ref{lem:Phi} 
\\ 
$\Hilb\Der(\cO_{D_X};C_X)=\Hilb\cO_{D_X}+t^{d_0}\Hilb\cO_{\tD}$ & & by Prop.~\ref{prop:descent}
\\
$\Hilb\cO_{D_X}=\Hilb\cO_X-t^{\deg\Lambda}\Hilb\cO_S$ & & by Lemma~\ref{lem:conductor}.
\\
$\Hilb\cO_{\tD}=\text{concrete formula}$ & & in Lemma~\ref{lem:normalisation}(a).
\end{tabular}
\end{center}
\noindent Hence we have
\begin{multline*}
\Hilb L=(1-t^{\mu k})\Hilb\cO_T+t^{(\mu-1)(k-v)}\big(\Hilb\theta_S-\Hilb\cO_S\big)\\
+t^{(\mu-1)(k-v)-a_1-a_2}\frac{(1-t^{(\mu-1)v})(1-t^{k-v})}{(1-t^{a_1})(1-t^{a_2})(1-t^v)(1-t^{2v})},
\end{multline*}
and (U1) gives $\Hilb\NAe$. (The shift $d_0=\deg\Lambda_3+e_\zeta$ of the second summand is consistent with the empty-triple-curve case, where $\deg\Lambda_3=0$ and $\bL$ is a unit.) This is the expression $\mathcal H_3+t^{d_0}H_{\tD}$ in Theorem~B. \qed

\subsection{Expanded Hilbert-series formulas}\label{sec:expanded}

The compact formulas in Theorems~A and~B can also be written as
\begin{align}
m=2:\quad \Hilb_t\NAe&=\frac{B_\mu(t)}{(1-t^a)(1-t^{\mu v})},\label{eq:expanded2}\\
m=3:\quad \Hilb_t\NAe&=\frac{B^{(3)}_\mu(t)}{(1-t^{a_1})(1-t^{a_2})(1-t^{\mu v})}.\label{eq:expanded3}
\end{align}
For $m=2$, the numerator is
\begin{align*}
B_\mu={}&\sum_{i=0}^{\mu-2}t^{v-k+iv}+\sum_{i=2}^{\mu-1}t^{-iv}
-t^{-a}-\sum_{i=0}^{\mu-2}t^{iv}\\
&-\sum_{j=1}^{\mu-2}\big(t^{jk-\mu v}+t^{jk-a}\big)
+\sum_{i=0}^{\mu-2}t^{(\mu-1)(k-v)-a+iv}.
\end{align*}
For $m=3$, set
\[
[n]_{t^v}=\frac{1-t^{nv}}{1-t^v},\qquad
\begin{bmatrix}\mu\\2\end{bmatrix}_{t^v}
=\frac{(1-t^{\mu v})(1-t^{(\mu-1)v})}{(1-t^v)(1-t^{2v})}.
\]
Writing $[n]=[n]_{t^v}$ and taking all sums over $l$ from $1$ to $2$, we have
\begin{align*}
B^{(3)}_\mu={}&t^{v-k}[\mu-1]+\sum_{i=2}^{\mu-1}t^{-iv}
-\sum_l t^{-a_l}-[\mu-1]\\
&-\sum_{j=1}^{\mu-2}\left(t^{jk-\mu v}+\sum_l t^{jk-a_l}\right)
+\sum_l t^{\lambda-a_l}[\mu-1]\\
&+\left(t^{\lambda-a_1-a_2}-t^{\mu(k-v)-a_1-a_2}\right)
\begin{bmatrix}\mu\\2\end{bmatrix}_{t^v}.
\end{align*}
Empty sums are zero. Appendix~\ref{sec:algebra} verifies these expansions and the limits used below.

\subsection{Proof of Corollary C}\label{sec:proofC}

Since $\NAe$ is finite-dimensional, its Hilbert series is a Laurent polynomial and its value at $t=1$ is $\dim\NAe$. Write each law as $\sum_ic_it^{\alpha_i}/\prod_j(1-t^{\beta_j})$ with exponents polynomial in $(a_i,v,k,\mu)$ and put $t=e^\epsilon$. If $\sum_ic_i\alpha_i^n=0$ for $n<r$, where $r$ is the number of denominator factors, the limit is $\frac1{r!}\sum_ic_i\alpha_i^r/\prod_j(-\beta_j)$. Appendix~\ref{sec:algebra} evaluates these moments by finite power sums
and gives
\begin{align}
m=2:\quad \Aecodim(f)&=R(w,d)=\frac{(k-v)(\mu-1)\big((\mu+1)k+(\mu-5)v-3a\big)}{6av},\label{eq:R2}\\
m=3:\quad \Aecodim(f)&=R(w,d)=\frac{(k-v)(\mu-1)}{24a_1a_2v}\Big[(\mu+1)(\mu+2)k^2\notag\\
&\qquad +(3\mu^2-19\mu+2)kv+2\mu(\mu+1)v^2\notag\\
&\qquad -4(\mu+1)(a_1+a_2)k-4(\mu-5)(a_1+a_2)v+12a_1a_2\Big].\label{eq:R3}
\end{align}

For $m=2$, the right-hand side of \eqref{eq:R2} is Mond's image Milnor number formula \cite{MonQH91}. For $m=3$, \eqref{eq:R3} is the specialisation of Ohmoto's formula \cite[Thm~6.20, Ex.~6.22]{Ohm16} to source weights $(a_1,a_2,v)$ and target degrees $(a_1,a_2,\mu v,k)$. Both formulas are also recovered in \cite{PPS21} and reproduced in \cite[Thms~2.4--2.5]{NBPS24}. Thus $\Aecodim(f)=\muI(f)$ in the two cases under consideration. They agree with the general weights-and-degrees expression of \cite[Thm~9.4]{KR}. The proof here uses only the two- and three-dimensional formulas of Mond and Ohmoto. The assertion about constancy of its Hilbert series follows from Theorems~A and~B. \qed

\begin{corollary}\label{cor:CM}
Let $m\in\{2,3\}$, and let $f\colon(\C^m,0)\to(\C^{m+1},0)$ be an $\Ae$-finite quasihomogeneous corank-one germ. Then $f$ has the Cohen--Macaulay property of \cite{FdBNBPS}: for some stable unfolding $F$ of $f$ with $r$ parameters, the module $M_y(F)$ of \cite{FdBNBPS} is Cohen--Macaulay of dimension $r$.
\end{corollary}

\begin{proof}
By Corollary~C, $f$ satisfies Mond's conjecture. For weighted homogeneous germs in the nice dimensions this is equivalent to the Cohen--Macaulay property \cite[Thm~6.7]{FdBNBPS}.
\end{proof}

\subsection{An example}\label{sec:example}

\begin{example}\label{ex:embedded}
Consider
\[
f=(x_1,x_2,y^3+x_1y,y^4+x_2y),\qquad
(a_1,a_2,v)=(2,3,1),\quad (\mu,k)=(3,4).
\]
The ordered double-point surface is smooth: its equations eliminate $x_1$ and $x_2$. For $D^3(f)$, the next equations are
\[
y_1+y_2+y_3=0,\qquad
y_1^2+y_2^2+y_3^2+y_1y_2+y_1y_3+y_2y_3=0.
\]
After eliminating $y_3$, the second equation becomes $y_1^2+y_1y_2+y_2^2=0$, a reduced nodal curve. The space $D^4(f)$ is empty, since $p[1234]=1$. Thus the Marar--Mond criterion gives $\Ae$-finiteness.

Put $s=y_1+y_2$ and $u=y_1y_2$, of weights $1$ and $2$. Then
\[
P_{12}=s^2-u+x_1,\qquad Q_{12}=s^3-2su+x_2,
\]
so $\cO_{\tD}=\C[s,u]$, and the normalisation map to $D_X$ is
\[
(X_1,X_2,Y,Z)=(u-s^2,\,2su-s^3,\,-su,\,u^2-s^2u).
\]
The branch equation and triple-point function are
\[
\bD=s^2-4u,\qquad\bL=s^2-u.
\]
The second formula follows by substituting $y_3=-s$ into $q[123]$. In these coordinates the Euler and cross-product fields are
\[
E=s\partial_s+2u\partial_u,\qquad
\zeta=-2\partial_s-s\partial_u
\]
with the orientation convention of Section~\ref{sec:freeness}. In particular,
\[
\zeta(\bD)=0,\qquad
E\wedge\zeta=-\bD\,\partial_s\wedge\partial_u.
\]
Their degrees are $0$ and $-1$, while $\deg\bL=2$. Theorem~B therefore gives
\[
\Der(\cO_{D_X};C_X)
=\cO_{D_X}E\oplus(s^2-u)\C[s,u]\zeta.
\]
Substitution in \eqref{eq:expanded3} yields
\[
B^{(3)}_3=t^{-2}(1-t^2)(1-t^3)^2,\qquad
\Hilb_t\NAe=t^{-2},\qquad \Aecodim(f)=\muI(f)=1.
\]

This example also illustrates why Lemma~\ref{lem:normalisation}(c) uses the reduced inverse image of $C_X$. The reduced ideal of that curve in the target is
\[
(X_2-2Y,\,3Z+X_1^2,\,27Y^2+4X_1^3).
\]
Its pullback to $\C[s,u]$ is generated by
\[
-s(s^2-4u),\quad (s^2-u)(s^2-4u),\quad
-(s^2-4u)(2s^2+u)^2,
\]
and hence is $(s^2-4u)(s,u)$. This ideal has an embedded point at the vertex, while its radical is the principal ideal $(s^2-4u)$.
\end{example}

\section{Beyond quasihomogeneity: families of unbounded corank-one multiplicity}\label{sec:beyond}

Fern\'andez de Bobadilla, Nu\~no-Ballesteros and Pe\~nafort Sanchis proved that Mond's conjecture follows from its quasihomogeneous case as soon as one has quasihomogeneous $\cA$-finite examples of unbounded multiplicity; in the form adapted to a fixed corank:

\begin{enumerate}[label=\textup{(FNP)}]
\item \cite[Thm~7.2]{FdBNBPS} (see also \cite[Thm~2.7]{NBPS24}). Let $(n,n+1)$ be nice dimensions and $r\ge1$. An $\cA$-finite germ of corank $r$ is $\cA$-equivalent to one of the form $(x_1,\dots,x_{n-r},\allowbreak f_{n-r+1},\dots,f_{n+1})$. Its \emph{corank-$r$ multiplicity} is $\min_{n-r+1\le i\le n+1}\operatorname{ord}f_i$. Suppose that for every natural number $M$ there exists a weighted homogeneous $\cA$-finite germ $h:(\C^n,0)\to(\C^{n+1},0)$ of corank $r$ and of corank-$r$ multiplicity at least $M$ for which Mond's conjecture holds. Then Mond's conjecture ($\Aecodim(f)\le\muI(f)$, with equality if $f$ is weighted homogeneous) holds for every $\cA$-finite germ $f:(\C^n,0)\to(\C^{n+1},0)$ of corank $r$.
\end{enumerate}

\begin{remark}[nice dimensions vs corank one]\label{rem:nicedims}
For $r=1$ the hypothesis that $(n,n+1)$ be nice dimensions is not needed. In \cite{FdBNBPS} it is used only to guarantee that (a) every $\cA$-finite germ admits a stable perturbation, and (b) all stable singularities appearing in a stable perturbation are quasihomogeneous. For corank-one germs both hold in every dimension: the singularities of a stable perturbation have corank at most one, and stable corank-one multigerms are multigerms of Morin singularities in general position, which are quasihomogeneous. Hence, for corank-one germs $(\C^n,0)\to(\C^{n+1},0)$ and every $n$, the image Milnor number is defined, and (FNP) as well as the inequality $\muI(f)\le\Aecodim(f)$ for weighted homogeneous $f$ \cite[\S3 and Cor.~6.3]{FdBNBPS} hold. The author is grateful to J.~J.~Nu\~no-Ballesteros for pointing out this argument.
\end{remark}

For $r=1$ the required germs are $(x,p,q)$ with $\operatorname{ord}p,\operatorname{ord}q\ge M$. In particular, \emph{homogeneous} $p,q$ (all weights $1$) of degrees $\mu,k\ge M$ qualify. The following proposition, which holds in every source dimension, supplies them.

\begin{proposition}[unbounded corank-one multiplicity]\label{prop:unbounded}
Let $m\ge2$ and let $\mu,k\ge2$ be integers with $\gcd(\mu,k)=1$. Let $V_\mu\times V_k$ be the space of pairs $(p,q)$ of homogeneous polynomials in $(x_1,\dots,x_{m-1},y)$ of degrees $\mu$ and $k$ (all variables of weight $1$). For $(p,q)$ in a dense Zariski-open subset of $V_\mu\times V_k$, the germ
\[
f=(x_1,\dots,x_{m-1},\,p,\,q):(\C^m,0)\to(\C^{m+1},0)
\]
is a quasihomogeneous corank-one germ with weights $(1,\dots,1)$ and degrees $(1,\dots,1,\mu,k)$, and it is $\Ae$-finite. Its corank-one multiplicity is $\min(\mu,k)$.
\end{proposition}

The proof uses the following converse to the multiple-point criterion, valid in every source dimension.

\begin{enumerate}[label=\textup{(MM$'$)}]
\item \cite[Thm~2.14(ii)]{MM89} Let $f:(\C^n,0)\to(\C^p,0)$, $n<p$, be a finite map germ of corank one. Then $f$ is finitely determined (equivalently, $\Ae$-finite \cite[Thm~6.1]{MNB20}) if and only if for each $j$ with $p-j(p-n)\ge0$ the multiple point scheme $D^j(f)$ (defined by the $(p-n+1)(j-1)$ divided differences, \S\ref{sec:multiple}) is either an ICIS of dimension $p-j(p-n)$ or empty, and for each $j$ with $p-j(p-n)<0$, $D^j(f)$ consists at most of the point $0$. (For $p=n+1=m+1$: $\dim D^j=m+1-j$ for $j\le m+1$, and $D^j\subseteq\{0\}$ for $j\ge m+2$.)
\end{enumerate}

\begin{proof}[Proof of Proposition~\ref{prop:unbounded}]
Let $c,c'$ be the coefficients of $y^\mu,y^k$ in $p,q$, respectively, and put $r=\min(\mu,k)$. On the dense open set $\{cc'\ne0\}$, the map $f$ is finite, $f^{-1}(0)=\{0\}$, and its corank-one multiplicity is $r$.

\emph{Step 1: the empty-space cutoff.} If $r=\mu$, then $p[1,\dots,\mu+1]=c\ne0$. If $r=k$, the corresponding divided difference of $q$ is $c'\ne0$. Hence
\[
D^j(f)=\emptyset\qquad\text{for every }j\ge r+1.
\]
It remains to consider $2\le j\le r$. In this range every defining divided difference has positive degree, so $0\in D^j(f)$.

\emph{Step 2: exclusion of the punctured axis.} We claim that $D^j(f)\cap\{x=0\}=\{0\}$ for $2\le j\le r$. A point of this intersection projects to $(0,y_1,y_2)\in D^2(f)$. The equations there are
\[
c\,h_{\mu-1}(y_1,y_2)=c'\,h_{k-1}(y_1,y_2)=0,
\qquad h_d(y_1,y_2)=\sum_{a+b=d}y_1^ay_2^b.
\]
If $y_1=y_2=y$, the first equation is $c\mu y^{\mu-1}=0$, so $y=0$. If $y_1\ne y_2$, neither can be zero, and the equations imply $(y_1/y_2)^\mu=(y_1/y_2)^k=1$, contrary to coprimality. Thus $y_1=y_2=0$. Since $D^j(f)$ is invariant under the action of $S_j$ (\S\ref{sec:multiple}), each pair $(y_1,y_i)$ satisfies the same first divided-difference equations, so all $y_i=0$.

\emph{Step 3: freedom of the divided differences.} Let
\[
B_j=\{(x,y_1,\dots,y_j):x\ne0\},\qquad
H^j_{(p,q)}=(p[1,\dots,i],q[1,\dots,i])_{2\le i\le j}.
\]
At a fixed $x\ne0$, the restriction $V_\mu\to\C[y]_{\le\mu}$ is surjective: in $p=\sum_{i=0}^\mu c_i(x)y^i$, a homogeneous form $c_i$ of degree $\mu-i$ can take any prescribed value at $x$. The same holds for $q$. For any nodes $y_1,\dots,y_j$, including repeated nodes, the Newton coefficients
\[
P[1],P[1,2],\dots,P[1,\dots,j]
\]
give an isomorphism from polynomials of degree at most $j-1$ to $\C^j$. Therefore, for every $z\in B_j$ and $2\le j\le r$, the evaluation map
\[
\Phi_z:V_\mu\times V_k\longrightarrow\C^{2(j-1)},\qquad
(p,q)\longmapsto H^j_{(p,q)}(z)
\]
is surjective.

\emph{Step 4: generic smoothness.} The incidence scheme
\[
\mathcal I_j=\{((p,q),z)\in(V_\mu\times V_k)\times B_j:
\Phi_z(p,q)=0\}
\]
is the kernel of a surjective morphism of vector bundles over $B_j$. It is consequently smooth and irreducible, of dimension
\[
\dim(V_\mu\times V_k)+(m-1+j)-2(j-1)
=\dim(V_\mu\times V_k)+m+1-j.
\]
If the projection $\pi_j:\mathcal I_j\to V_\mu\times V_k$ is not dominant, remove the closure of its image. If it is dominant, generic smoothness in characteristic zero \cite[III, Cor.~10.7]{Har77} gives a dense open set over which it is smooth, of relative dimension $m+1-j$. Thus there is a dense open $U_j$ such that $D^j(f)\cap B_j$ is smooth of pure dimension $m+1-j$ or empty for $(p,q)\in U_j$. It is empty when $m+1-j<0$.

\emph{Step 5: the complete intersections at the vertex.} Choose
\[
(p,q)\in\{cc'\ne0\}\cap\bigcap_{j=2}^{r}U_j.
\]
For $2\le j\le r$, Step~2 identifies $D^j(f)\setminus\{0\}$ with $D^j(f)\cap B_j$. The scheme $D^j(f)$ is a nonempty cone defined by $2(j-1)$ equations in $\C^{m-1+j}$. Write $e=m+1-j$. If $e\ge1$, Krull's height theorem gives dimension at least $e$ for every irreducible component. Each component meets the punctured cone, where the dimension is $e$, so all components have dimension exactly $e$. The defining ideal therefore has height $2(j-1)$, and its generators form a regular sequence. Thus $D^j(f)$ is an ICIS of dimension $e$. If $e=0$, the punctured cone is smooth of dimension zero and hence empty: any nonzero point would lie on a one-dimensional $\C^*$-orbit. The scheme is then a zero-dimensional complete intersection. If $e<0$, Step~4 gives $D^j(f)\subseteq\{0\}$ set-theoretically. Together with the empty spaces of Step~1, these are precisely the alternatives in (MM$'$). Hence $f$ is $\Ae$-finite.
\end{proof}

\begin{remark}\label{rem:viability}
Coprimality is used only in Step~2, to exclude multiple points on the punctured axis $\{x=0\}$, and it cannot simply be dropped. Let $\nu=\gcd(\mu,k)\ge2$, $cc'\ne0$, and let $\rho$ be a primitive $\nu$-th root of unity. Then $f(0,\rho y)=f(0,y)$, so every point of the image of the punctured axis has at least $\nu$ preimages, at each of which $f$ is an immersion. On $\{x=0\}$ one has $p_{x_i}=b_iy^{\mu-1}$, $q_{x_i}=b'_iy^{k-1}$, $p_y=\mu cy^{\mu-1}$ and $q_y=kc'y^{k-1}$ for constants $b_i,b'_i$, and all of these are multiplied by $\rho^{-1}$ under $y\mapsto\rho y$. Hence the tangent hyperplanes of the branches at the points $(0,\rho^jy_0)$, $j=0,\dots,\nu-1$, have conormal vectors $(-\rho^{-j}\gamma,\beta^\perp)$, where $\gamma\in(\C^{m-1})^*$ and $\beta^\perp\in(\C^2)^*$ do not depend on $j$. These conormals span at most a plane. Thus, for $\nu\ge3$, the branches are not in general position and the multigerm over $f(0,y_0)$ is not stable. See \S\ref{sec:stable}. By (MG), the germ cannot be $\Ae$-finite. For $\nu=2$ the two conormals are independent as soon as $\gamma\ne0$. In our computations generic homogeneous germs with $(\mu,k)=(2,4)$ or $(4,6)$, for $m=2,3$, and with $(\mu,k)=(2,6)$, for $m=3$, were $\Ae$-finite. We do not know whether $\gcd(\mu,k)\le2$ suffices in general. For nontrivial source weights the coefficient-surjectivity argument of Step~3 may also fail, because homogeneous forms of the required weights need not exist.

Proposition~\ref{prop:unbounded} supplies the large-order nonlinear components needed for the fixed-corank reduction theorem \cite[Thm~7.2]{FdBNBPS}. The unrestricted reduction theorem \cite[Thm~7.1]{FdBNBPS} instead requires large orders for all coordinate functions. Pe\~nafort Sanchis observed that no $\cA$-finite family with arbitrarily high multiplicity in this sense was known beyond germs of curves \cite{Pe20}. Our construction has $m-1$ linear coordinates and addresses the former theorem. See \cite[\S2.5]{NBPS24} for the broader problem of constructing families of unbounded multiplicity.
\end{remark}

\begin{proof}[Proof of the Main Theorem]
$(3,4)$ are nice dimensions. For each $M\ge2$, Proposition~\ref{prop:unbounded}, with $(\mu,k)=(M,M+1)$, gives a quasihomogeneous $\Ae$-finite corank-one germ $h_M=(x_1,x_2,p,q)$ of corank-one multiplicity $M$. By Corollary~C, $\Aecodim(h_M)=\muI(h_M)$, so Mond's conjecture holds for $h_M$. We may therefore apply (FNP) with $n=3$ and $r=1$, which gives the inequality for every $\Ae$-finite corank-one germ. For quasihomogeneous germs equality is Corollary~C (Remarks~\ref{rem:normalform} and~\ref{rem:normalisation}).
\end{proof}

\begin{corollary}[reduction in all dimensions]\label{cor:reduction}
Let $m\ge2$. The following are equivalent.
\begin{enumerate}[label=\textup{(\roman*)}]
\item Mond's conjecture holds for all $\Ae$-finite corank-one germs $(\C^m,0)\to(\C^{m+1},0)$.
\item $\Aecodim(f)=\muI(f)$ for all $\Ae$-finite quasihomogeneous corank-one germs $(\C^m,0)\to(\C^{m+1},0)$.
\item For infinitely many $M$ there is an $\Ae$-finite germ $h_M=(x,p,q)$, with $p,q$ homogeneous of degrees $M$ and $M+1$, such that $\Aecodim(h_M)\le\muI(h_M)$.
\end{enumerate}
Moreover, if $m\le14$, then $\muI(h_M)=R(w,d)$ with $w=(1,\dots,1)$ and $d=(1,\dots,1,M,M+1)$, the value of \cite[Thm~9.4]{KR}.
\end{corollary}

\begin{proof}
By Remark~\ref{rem:nicedims}, the results of \cite{FdBNBPS} used below apply to corank-one germs in every dimension, although $(m,m+1)$ is a pair of nice dimensions only for $m\le14$ \cite{Mat71}. (i)$\Rightarrow$(ii) is part of the conjecture. (ii)$\Rightarrow$(iii): Proposition~\ref{prop:unbounded} with $(\mu,k)=(M,M+1)$ provides $\Ae$-finite germs $h_M$ for every $M\ge2$. (iii)$\Rightarrow$(i): for weighted homogeneous corank-one germs $\muI(h)\le\Aecodim(h)$ \cite[\S3 and Cor.~6.3]{FdBNBPS} (Remark~\ref{rem:nicedims}), so (iii) gives $\Aecodim(h_M)=\muI(h_M)$, i.e.\ Mond's conjecture for $h_M$. The corank-one multiplicity of $h_M$ is $M$, so (FNP) applies. The last statement is \cite[Thm~9.4]{KR}, $\muI(f)=R(w,d)$ for quasihomogeneous $\cA$-finite $f$ with $m\le14$.
\end{proof}

\section{What changes for \texorpdfstring{$m>3$}{m>3}?}\label{sec:outlook}

The proof of Theorem~B replaces the image $X$ by its double-point surface $D_X$ and then by the normalisation $\tD=D^2(f)/S_2$. The triple-point conductor governs the descent from $\tD$ back to $D_X$, and each step is extended across the vertex by a depth-two argument. For $m=4$ the same scheme acquires one more level: the triple-point surface of $X$, its normalisation $D^3(f)/S_3$, and the quadruple-point conductor. In computer experiments with unweighted germs, this recursive structure leads to a conjectural closed formula for $\Hilb_t\NAe$. It was checked on the modules themselves for $(\mu,k)=(3,4)$ and $(3,5)$, and, as an identity of formulas, its value at $t=1$ agrees with \cite[Thm~9.4]{KR} for $3\le\mu\le7$, $\mu<k\le\mu+6$.
The main step that does not carry over is the descent from $\tD$ to $D_X$ at the vertex. For $m=4$ the normalisation $\tD$ is three-dimensional, and the module of its derivations tangent to the branch divisor is no longer free. In the examples we computed, it is not even Cohen--Macaulay, it has depth two. The surjectivity of the descent map at the origin, which the depth argument of Proposition~\ref{prop:descent} gives for free when $m=3$, would then need a separate lifting argument. By Theorem~D, the corresponding upper bound for the generic homogeneous germs of degrees $(M,M+1)$ would already give Mond's conjecture for all corank-one germs $(\C^4,0)\to(\C^5,0)$.


\appendix
\section{Algebra for the Hilbert-series formulas}\label{sec:algebra}

We verify the expansions of Section~\ref{sec:expanded} and their values at $t=1$ by geometric sums and power sums. This can easily be verified by computer algebra, but for completeness we provide a human computation.

\subsection{From the structural formulas to the finite sums}

Set
\[
A_i=t^{a_i},\qquad V=t^v,\qquad K=t^k,\qquad
P=V^{\mu-1},\qquad Q=K^{\mu-1},\qquad
b=\sum_{i=1}^{m-1}A_i^{-1}.
\]
Multiplying $\mathcal H_m$ by $\prod_i(1-A_i)(1-VP)$ gives
\begin{align*}
&\frac{1-VP}{1-V}
\left(\frac1{VP}+\frac1K-\frac1V
+\frac QP\left(b+\frac1V-1\right)\right)\\
&\hspace{15mm}+\frac{1-KQ}{1-K}
\left(1-b-\frac1{VP}-\frac1K\right).
\end{align*}
Collecting the terms independent of $b$ and those proportional to $b$ rewrites this as
\begin{equation}\label{eq:geometric-numerator}
\left(\frac VK-1\right)\frac{1-P}{1-V}
+\frac{P^{-1}-V^{-1}}{1-V}-b
-\frac{K-Q}{1-K}\left(\frac1{VP}+b\right)
+\frac QP\,b\frac{1-P}{1-V}.
\end{equation}
The identities
\[
\frac{1-P}{1-V}=\sum_{i=0}^{\mu-2}V^i,\qquad
\frac{P^{-1}-V^{-1}}{1-V}=\sum_{i=2}^{\mu-1}V^{-i},\qquad
\frac{K-Q}{1-K}=\sum_{j=1}^{\mu-2}K^j
\]
give precisely $B_\mu$ when $m=2$, and the first six groups of terms in $B^{(3)}_\mu$ when $m=3$.

For $m=3$, the additional summand $t^{d_0}H_{\tD}$ contributes, after multiplication by the same denominator,
\[
\frac{Q}{PA_1A_2}\left(1-\frac KV\right)
\frac{(1-VP)(1-P)}{(1-V)(1-V^2)}
=\left(t^{\lambda-a_1-a_2}-t^{\mu(k-v)-a_1-a_2}\right)
\begin{bmatrix}\mu\\2\end{bmatrix}_{t^v}.
\]
This proves \eqref{eq:expanded2} and \eqref{eq:expanded3}, including $\mu=2$.

\subsection{The moments at \texorpdfstring{$t=1$}{t=1}}

For a Laurent polynomial $B(t)=\sum_\alpha c_\alpha t^\alpha$, write
\[
C_j(B)=\sum_\alpha c_\alpha\alpha^j,
\]
where $C_0(B)=\sum_\alpha c_\alpha$. With $t=e^\epsilon$, one has
$B(e^\epsilon)=\sum_{j\ge0}C_j(B)\epsilon^j/j!$.
Thus if $C_j(B)=0$ for $j<r$, then
\begin{equation}\label{eq:moment-limit}
\lim_{t\to1}\frac{B(t)}{\prod_{i=1}^r(1-t^{\beta_i})}
=\frac{C_r(B)}{r!\prod_{i=1}^r(-\beta_i)}.
\end{equation}

All moments required here follow from
\begin{gather*}
\sum_{i=0}^{N}1=N+1,\qquad
\sum_{i=0}^{N}i=\frac{N(N+1)}2,\\
\sum_{i=0}^{N}i^2=\frac{N(N+1)(2N+1)}6,\qquad
\sum_{i=0}^{N}i^3=\frac{N^2(N+1)^2}4.
\end{gather*}
For the Gaussian binomial, use
\[
\begin{bmatrix}\mu\\2\end{bmatrix}_q
=\sum_{0\le i<j\le\mu-1}q^{i+j-1}.
\]
Writing $\ell=i+j-1$, its first three moments are
\begin{align*}
\sum_{i<j}1&=\frac{\mu(\mu-1)}2,\\
\sum_{i<j}\ell&=\frac{\mu(\mu-1)(\mu-2)}2,\\
\sum_{i<j}\ell^2&=\frac{\mu(\mu-1)(\mu-2)(7\mu-11)}{12}.
\end{align*}
No cubic Gaussian moment is needed: its two occurrences in $B^{(3)}_\mu$ have opposite signs, so the cubic term cancels in their difference.

For completeness, the moment substitution can be written uniformly. For $j\ge0$ put
\[
T_j(c,N,h)=\sum_{i=0}^{N}(c+ih)^j
=\sum_{l=0}^j\binom jl c^{j-l}h^l\sum_{i=0}^{N}i^l,
\]
with $T_j(c,-1,h)=0$. Define
\begin{align*}
F_j={}&T_j(v-k,\mu-2,v)+T_j(-2v,\mu-3,-v)
-\sum_l(-a_l)^j-T_j(0,\mu-2,v)\\
&-T_j(k-\mu v,\mu-3,k)
-\sum_l T_j(k-a_l,\mu-3,k)
+\sum_l T_j(\lambda-a_l,\mu-2,v).
\end{align*}
Here the sum over $l$ runs through the $m-1$ weights $a_l$, and zeroth powers contribute $1$. For $m=2$, $C_j(B_\mu)=F_j$. For $m=3$, setting $d_0=\lambda-a_1-a_2$ and $h=k-v$, one has
\[
C_j(B^{(3)}_\mu)=F_j+
\sum_{0\le i<l\le\mu-1}
\left((d_0+v(i+l-1))^j-(d_0+h+v(i+l-1))^j\right).
\]
Substitution of the power sums above gives
\begin{gather*}
C_0(B_\mu)=C_1(B_\mu)=0,\\
C_2(B_\mu)=\frac{\mu(\mu-1)(k-v)}3
\big((\mu+1)k+(\mu-5)v-3a\big),
\end{gather*}
and
\begin{gather*}
C_0(B^{(3)}_\mu)=C_1(B^{(3)}_\mu)=C_2(B^{(3)}_\mu)=0,\\
C_3(B^{(3)}_\mu)=-\frac{\mu(\mu-1)(k-v)}4\,\mathcal P_3,
\end{gather*}
where
\begin{align*}
\mathcal P_3={}&(\mu+1)(\mu+2)k^2+(3\mu^2-19\mu+2)kv
+2\mu(\mu+1)v^2\\
&-4(\mu+1)(a_1+a_2)k-4(\mu-5)(a_1+a_2)v+12a_1a_2.
\end{align*}
Applying \eqref{eq:moment-limit} with $(\beta_1,\beta_2)=(a,\mu v)$ or $(\beta_1,\beta_2,\beta_3)=(a_1,a_2,\mu v)$ gives \eqref{eq:R2} and \eqref{eq:R3}.

\subsection*{Corank two}
For $m=3$ the only other corank is two, $f=(x,p,q,r)$ with $p,q,r\in\C[x,y,z]$. Part of the argument survives. Corank-two points have codimension $6$ in the source, so for an $\Ae$-finite germ they occur only at the origin, and over $X\setminus\{0\}$ the multigerm is again of type $A_0^j$ or $S_0$. Thus the local models of Section~\ref{sec:local} are unchanged. Since $\cO_X$ and $\cO_S$ are Gorenstein, the conductor $\Hom_{\cO_X}(\cO_S,\cO_X)$ is an invertible $\cO_S$-module. Hence $\fc=\lambda\cO_S$ for a reduced equation $\lambda$ of the double-point locus \cite{MP89}, with $\deg\lambda=\deg G+\sum_iw_i-\sum_jd_j$ determined by the weights and degrees. We expect Lemma~\ref{lem:conductor} and Theorem~\ref{thm:rigidity} to extend with essentially the same proofs. On the other side of the conjecture, the formula $R(w,d)$ of \cite{Ohm16,KR} does not depend on the corank. The obstacles are the following.
(i) In corank one, $D^2(f)$, $\Sigma$ and $D^3(f)$ are complete intersections given by explicit divided differences in one variable. This is what gives the normality of $D^2(f)$ (Lemma~\ref{lem:D2normal}), the reducedness of $\Lambda$ and $\Lambda_3$ via resultants, and the presentation of $\tD=D^2(f)/S_2$ as a complete intersection in $\C^4$ (Lemma~\ref{lem:normalisation}). In corank two the multiple-point spaces are defined through Fitting ideals \cite{MP89,NBPS} and are not cut out by divided differences. We know of no general result that $D^2(f)$ is normal or that $D^2(f)/S_2$ is a complete intersection, so the normalisation of $D_X$ and the triple-point conductor lose their explicit description.
(ii) The second generator $\zeta$ of Proposition~\ref{prop:freeness} is the cross product of the gradients of the three equations of $\tD$ and its branch curve in $\C^4$. Without such a presentation there is no candidate field, and the derivations of $\tD$ tangent to the branch curve need not form a free module. The descent of Proposition~\ref{prop:descent} then has no explicit target.
(iii) For arbitrary corank-two germs, \cite[Thm~7.2]{FdBNBPS} with $r=2$ would require quasihomogeneous $\cA$-finite corank-two germs of unbounded corank-two multiplicity. The proof of Proposition~\ref{prop:unbounded} rests on the Marar--Mond criterion (MM$'$), which is specific to corank one, and we do not know such a family.

So

\end{document}